\RequirePackage{fix-cm}

\documentclass[smallextended]{svjour3}       %
\smartqed  
\usepackage{csquotes}
\usepackage{graphicx}
\usepackage{amsfonts}
\usepackage{algorithm}
\usepackage{amsmath}
\usepackage{geometry}
\usepackage{color}
\usepackage{amssymb}

\usepackage{algpseudocode}

\usepackage[hidelinks]{hyperref}
\usepackage{cite}

\begin{document}

\title{Polynomial interpolation--regression on the sphere}

\titlerunning{Polynomial approximation on the sphere}        

\author{Francesco Dell'Accio \and 
        Federico Nudo \and 
         Teresa E. Pérez \and 
         Miguel A. Piñar
}

\institute{Francesco Dell'Accio \at
             Department of Mathematics and Computer Science, University of Calabria, Rende (CS), Italy\\ 
             \email{francesco.dellaccio@unical.it} 
         \and 
            Federico Nudo (corresponding author) \at
              Department of Mathematics and Computer Science, University of Calabria, Rende (CS), Italy\\
              \email{federico.nudo@unical.it}
                \and 
              Teresa E. Pérez \at
              Institute of Mathematics IMAG and Department of Applied Mathematics, University of Granada, Granada, Spain \\
              \email{tperez@ugr.es}
              \and 
              Miguel A. Piñar \at
              Department of Applied Mathematics, University of Granada, Granada, Spain \\
              \email{mpinar@ugr.es}
}

\date{Received: date / Accepted: date}

\maketitle

\begin{abstract}
We introduce an interpolation--regression operator for polynomial approximation on the unit sphere $\mathbb{S}^2$ from discrete samples. The approximant is a spherical polynomial of degree $r$ which interpolates the data on a prescribed subset of nodes and uses the remaining sampling nodes to minimize the residual in a least squares sense. Under natural rank assumptions on the associated Vandermonde matrices, the approximant is unique and is characterized by an orthogonality condition with respect to the discrete inner product on the sampling set. We then focus on the case in which the sampling and interpolation nodes are antipodally symmetric. In this setting, when the polynomial is expressed in real spherical harmonics, the constrained problem can be decomposed into independent even and odd components. In the same framework, we prove equivariance under the antipodal map and, more generally, under orthogonal transformations preserving the node sets. We also consider spherical designs. In this case, the normal matrix is a scalar matrix. Consequently,  the spectral condition number of the associated KKT matrix can be written explicitly. Numerical experiments in both antipodal and non-antipodal settings illustrate the effectiveness of the proposed method.

\end{abstract}

\section{Introduction}
The reconstruction of functions on the unit sphere $\mathbb{S}^2$ is a classical problem in applied and computational mathematics~\cite{BSI:MWR:1996,CAO:OUP:1998,SAA:ACM:2001,FIF:TRA:2021}. Problems of this type arise in geophysics, meteorology, climate modeling, astrophysics, and computer graphics, where data are naturally sampled on spherical geometries. In many such situations, the available information is reduced to a finite set of samples. This has led to the development of several approximation techniques; see, e.g.,~\cite{CPA:JAPP:2000,IBP:CONAPP:2001,SAA:ACM:2001,PIO:SINUM:2003,PIO:ACOM:2004,PSO:DEC:2012,SFM:ACM:2015,IOS:JOMP:2025}. Among these methods, polynomial approximation is particularly attractive for its simple algebraic structure and for the efficient evaluation of derivatives, integrals, and spherical differential operators. Polynomial interpolation on $\mathbb{S}^2$ is uniquely solvable only for suitable node configurations. These configurations, usually called \emph{unisolvent} sets, consist of $(n+1)^2$ points for which the associated Vandermonde matrix is nonsingular. Explicit constructions of such sets have been introduced in~\cite{PIO:SINUM:2003,PIO:ACOM:2004}. However, in many practical applications, the sampling points are fixed a priori and do not exhibit any special geometric structure. In particular, they are typically not unisolvent, so that pure interpolation may become either ill-posed or numerically unstable. In this situation, least squares approximation provides a flexible alternative for arbitrary sampling sets. However, used alone, it does not exploit the local accuracy that interpolation at well chosen nodes may provide. The aim of this work is to introduce an interpolation--regression operator for polynomial approximation on the sphere. The approximant interpolates the data on a prescribed subset of nodes and uses the remaining sampling nodes to minimize the residual in a least squares sense. Once a basis of the polynomial space is fixed, the coefficients of the interpolation--regression polynomial are obtained from a constrained least squares problem, which is written as a Karush--Kuhn--Tucker (KKT) linear system. Our construction is related to the constrained \emph{mock-Chebyshev least squares} approach introduced in the one dimensional setting in~\cite{OTC:JCAM:2015}. Starting from a grid of $n+1$ equispaced nodes, that method selects a subset of $m=O(\sqrt{n})$ nodes approximating a Chebyshev distribution and computes a polynomial that interpolates the data on this subset while minimizing the residual over the full sampling set in a least squares sense. This strategy has subsequently been extended to several multidimensional settings, including tensor product constructions on rectangular domains~\cite{GOT:AML:2022}, as well as related approaches on triangular, circle, and disk domains~\cite{AMI:DRNA:2024,ANI:ACOM:2025,cruz2026mixed}. Recently, this method has also been applied to the histopolation setting~\cite{bruni2025polynomial}. Motivated by these developments, we extend this approach to the sphere.

The paper is organized as follows.  In Section~\ref{sec:interp_sphere}, we introduce the interpolation--regression problem on $\mathbb{S}^2$ and study its main algebraic properties. In particular, we prove the nonsingularity of the associated KKT system under natural rank assumptions, determine its inertia, and characterize the approximant through a discrete orthogonality condition. We also analyze the role of antipodal symmetry, showing that, when the polynomial is expressed in real spherical harmonics, the problem can be decomposed into even and odd components. Section~\ref{sec:designs} is devoted to spherical designs. In this case, the spectral condition number of the associated KKT matrix can be characterized explicitly. In Section~\ref{secQuasiOp}, we establish a quasi-optimality result under a discrete norming assumption on the sampling set. In Section~\ref{sec:algorithm}, we present a constructive strategy for generating admissible interpolation subsets in the antipodally symmetric setting. Finally, in Section~\ref{sec:numtest} we report numerical experiments in both antipodal and non-antipodal configurations.

\section{Interpolation--regression on the sphere}
\label{sec:interp_sphere}

Let
\[
X_N=\left\{ \boldsymbol{x}_1,\dots,\boldsymbol{x}_N \right\}\subset \mathbb{S}^2
\]
be a set of sampling nodes, and assume that the values of a function
\[
f:\mathbb{S}^2\to\mathbb{R}
\]
are known at the nodes of $X_N$. Although the construction works for sampling sets of arbitrary size, for simplicity we assume that $N=(n+1)^2$ for some $n\in\mathbb{N}$. Let $r<n$ and denote by $\Pi_r\left( \mathbb{S}^2 \right)$
the space of spherical polynomials of degree at most $r$, with
\[
R:=(r+1)^2=\dim\left( \Pi_r\left( \mathbb{S}^2 \right) \right).
\]
Given a basis 
\begin{equation}\label{basis}
\mathcal{B}_r=\left\{b_1,\dots,b_R\right\}  
\end{equation}
of $\Pi_r\left( \mathbb{S}^2 \right)$, we consider the associated Vandermonde matrix
\begin{equation}\label{vander}
    V_{N,r}=
\left[
b_j\left( \boldsymbol{x}_i \right)
\right]_{\substack{i=1,\dots,N \\ j=1,\dots,R}}
\in\mathbb{R}^{N\times R},
\end{equation}
and assume that
\begin{equation}\label{eq:rank_condition_section}
\operatorname{rank}\left( V_{N,r} \right)=R.
\end{equation}
A convenient choice for $\mathcal{B}_r$ is given by the real spherical harmonics. In spherical coordinates $(\varphi,\theta)\in[0,\pi]\times[0,2\pi)$, they are defined by
\begin{equation}\label{real_armsfbas}
    \mathcal{B}_r=
\left\{
u_{\ell,m}(\varphi,\theta)\,:\,
0\le \ell\le r,\quad -\ell\le m\le \ell
\right\},
\end{equation}
where
\begin{equation}\label{real_armsf}
    u_{\ell,m}(\varphi,\theta)=
\begin{cases}
\sqrt{2}N_{\ell,m}P_\ell^m(\cos\varphi)\cos(m\theta),
& 1\le m\le \ell,\\[6pt]
N_{\ell,0}P_\ell^0(\cos\varphi),
& m=0,\\[6pt]
\sqrt{2}N_{\ell,|m|}P_\ell^{|m|}(\cos\varphi)\sin(|m|\theta),
& -\ell\le m\le -1.
\end{cases}
\end{equation}
Here $P_\ell^m$ denotes the associated Legendre polynomial~\cite{Abramowitz:1948:HOM} and
\[
N_{\ell,m}=
\sqrt{\frac{2\ell+1}{4\pi}\frac{(\ell-m)!}{(\ell+m)!}}.
\]
With this normalization, the basis $\mathcal{B}_r$ is orthonormal in $L^2\left(\mathbb{S}^2\right)$ with respect to the measure 
\[
d\omega = \sin(\varphi) d\varphi d\theta,
\]
that is
\begin{equation}\label{ortcond1}
   \langle u_{\ell,m},u_{\ell',m'}\rangle_{L^2\left(\mathbb{S}^2\right)}=\int_{\mathbb{S}^2}
u_{\ell,m}(\boldsymbol{x})u_{\ell',m'}(\boldsymbol{x})d\omega(\boldsymbol{x})
=
\delta_{\ell\ell'}\delta_{mm'}, \quad \boldsymbol{x}=(\varphi,\theta).
\end{equation}
We now select a subset of nodes on which interpolation can be imposed. Let $m<r$ and define
\[
M:=(m+1)^2=\dim\left( \Pi_m\left( \mathbb{S}^2 \right) \right).
\]
We select a subset
\[
X_M=\left\{ \boldsymbol{x}_{1}^{\prime},\dots,\boldsymbol{x}_{M}^{\prime} \right\}\subset X_N,
\]
and denote by $V_{M,r}\in\mathbb{R}^{M\times R}$ the matrix obtained by restricting $V_{N,r}$ to the rows corresponding to $X_M$. We assume that
\begin{equation}\label{rankM}
    \operatorname{rank}\left(V_{M,r}\right)=M.
\end{equation}
Although the construction does not require $M$ to be of the form $(m+1)^2$, we adopt this convention throughout the paper. Indeed, this value of $M$ coincides with the number of degrees of freedom of $\Pi_m\left(\mathbb{S}^2\right)$ and therefore matches the number of conditions required for polynomial interpolation. This choice is used in Section~\ref{secQuasiOp} to establish convergence results under stronger interpolation assumptions.

We seek a polynomial $\hat{p}_r[f]\in \Pi_r\left( \mathbb{S}^2 \right)$ that satisfies the interpolation conditions
\[
\hat{p}_r[f]\left(\boldsymbol{x}_{j}^{\prime}\right)=f\left(\boldsymbol{x}_{j}^{\prime}\right),
\qquad j=1,\dots,M,
\]
while approximating the data over $X_N\setminus X_M$ in the least squares sense. This leads to the following constrained approximation problem
\begin{equation}\label{eq:constrained_problem_sphere}
\min_{p_r\in \Pi_r\left( \mathbb{S}^2 \right)}
\sum_{i=1}^N
\left|
p_r(\boldsymbol{x}_i)-f(\boldsymbol{x}_i)
\right|^2
\quad
\text{subject to}
\quad
p_r\left(\boldsymbol{x}_{j}^{\prime}\right)=f\left(\boldsymbol{x}_{j}^{\prime}\right),
\quad j=1,\dots,M.
\end{equation}
We express this polynomial in the basis~\eqref{basis}, namely
\begin{equation}\label{eq:cls_problem}
    \hat{p}_r[f]=\sum_{j=1}^R c_j b_j,
\end{equation}
with the unknowns
\[\boldsymbol{c}=\left[c_1,\dots,c_R\right]^{\top}\in\mathbb{R}^R.\]
Introducing the data vectors
\[
\boldsymbol{f}_N=\left[f\left(\boldsymbol{x}_1\right),\dots,f\left(\boldsymbol{x}_N\right)\right]^\top,
\quad
\boldsymbol{f}_M=\left[f\left(\boldsymbol{x}_{1}^{\prime}\right),\dots,f\left(\boldsymbol{x}_{M}^{\prime}\right)\right]^\top,
\]
problem~\eqref{eq:constrained_problem_sphere} can be written in the matrix form as
\begin{equation}\label{eq:matrix_problem_sphere}
\min_{\boldsymbol{c}\in\mathbb{R}^R}
\left\|V_{N,r}\boldsymbol{c}-\boldsymbol{f}_N\right\|_2^2
\quad
\text{subject to}
\quad
V_{M,r}\boldsymbol{c}=\boldsymbol{f}_M.
\end{equation}
These conditions lead to the KKT system
\begin{equation}\label{eq:kkt_sphere}
\begin{bmatrix}
V_{N,r}^\top V_{N,r} & V_{M,r}^\top \\
V_{M,r} & 0
\end{bmatrix}
\begin{bmatrix}
\boldsymbol{c} \\ \boldsymbol{\gamma}
\end{bmatrix}
=
\begin{bmatrix}
V_{N,r}^\top \boldsymbol{f}_N \\
\boldsymbol{f}_M
\end{bmatrix},
\end{equation}
where $\boldsymbol{\gamma}\in\mathbb{R}^M$ is the vector of Lagrange multipliers. The next result establishes solvability of the KKT linear system and describes its inertia.

\begin{theorem}\label{thm1}
Let $X_N \subset \mathbb{S}^2$ be a set of sampling nodes, and let $X_M \subset X_N$ be a subset of interpolation nodes. Let
\[
V_{N,r}\in\mathbb{R}^{N\times R},
\qquad
V_{M,r}\in\mathbb{R}^{M\times R},
\]
be the corresponding Vandermonde matrices. Assume that
\begin{itemize}
    \item[(i)] $\operatorname{rank}\left( V_{N,r} \right)=R$,
    \item[(ii)] $\operatorname{rank}\left( V_{M,r} \right)=M$.
\end{itemize}
Then the KKT matrix
\begin{equation}
\label{eq:KKT_matrix}
\mathcal{K}=
\begin{bmatrix}
V_{N,r}^\top V_{N,r} & V_{M,r}^\top \\[4pt]
V_{M,r} & \boldsymbol{0}
\end{bmatrix}
\in\mathbb{R}^{(R+M)\times(R+M)}
\end{equation}
is nonsingular. Moreover, $\mathcal{K}$ has exactly $R$ positive eigenvalues and $M$ negative eigenvalues.
\end{theorem}

\begin{proof}
Set
\[
A=V_{N,r}^\top V_{N,r},
\qquad
B=V_{M,r}.
\]
By assumption~$(i)$, $A$ is symmetric positive definite. Indeed, for every $\boldsymbol{x}\in\mathbb{R}^R\setminus\{\boldsymbol{0}\}$, we have
\[
\boldsymbol{x}^\top A \boldsymbol{x}=\boldsymbol{x}^\top V_{N,r}^\top V_{N,r} \boldsymbol{x}
=
\left\| V_{N,r}\boldsymbol{x} \right\|_2^2
>0.
\]
With this notation, we can write
\begin{equation}\label{mat_K}
    \mathcal{K}=
\begin{bmatrix}
A & B^\top\\
B & \boldsymbol{0}
\end{bmatrix}.
\end{equation}
To prove the nonsingularity of $\mathcal{K}$, let $[\boldsymbol{c},\boldsymbol{\gamma}]^{\top}\in\mathbb{R}^R\times\mathbb{R}^M$ satisfy
\[
\mathcal{K}
\begin{bmatrix}
\boldsymbol{c}\\
\boldsymbol{\gamma}
\end{bmatrix}
=
\begin{bmatrix}
\boldsymbol{0}\\
\boldsymbol{0}
\end{bmatrix}.
\]
Then, we have
\begin{eqnarray}\label{saaa}
A\boldsymbol{c}+B^\top\boldsymbol{\gamma}&=&\boldsymbol{0}\\
      B\boldsymbol{c}&=&\boldsymbol{0}. \label{saa2}
\end{eqnarray}
Multiplying~\eqref{saaa} on the left by $\boldsymbol{c}^\top$ gives
\begin{equation}\label{pri:eq}
    \boldsymbol{c}^\top A \boldsymbol{c}
+
\boldsymbol{c}^\top B^\top \boldsymbol{\gamma}
=0, 
\end{equation}
and, by~\eqref{saa2}
\begin{equation}\label{sec:eq}
    \boldsymbol{c}^\top B^\top \boldsymbol{\gamma}
=
(B\boldsymbol{c})^\top \boldsymbol{\gamma}
=0.
\end{equation}
Substituting~\eqref{sec:eq} into~\eqref{pri:eq} gives
\[
\boldsymbol{c}^\top A \boldsymbol{c}=0.
\]
Since $A$ is positive definite, this implies $\boldsymbol{c}=\boldsymbol{0}$. Substituting into~\eqref{saaa} yields
\[
B^\top \boldsymbol{\gamma}=\boldsymbol 0.
\]
By assumption~$(ii)$, $B$ has full row rank, so $\boldsymbol{\gamma}=0$. Hence $\mathcal{K}$ is nonsingular.

To determine the inertia of $\mathcal{K}$, consider the Schur complement
\begin{equation}\label{ShurComp}
    S=BA^{-1}B^\top.
\end{equation}
Since $A^{-1}$ is positive definite and $B$ has full row rank, for every $\boldsymbol{y}\neq\boldsymbol{0}$, we have
\[
\boldsymbol{y}^\top S \boldsymbol{y}=\boldsymbol{y}^\top BA^{-1}B^\top \boldsymbol{y}
=
(B^\top \boldsymbol{y})^\top A^{-1}(B^\top \boldsymbol{y})
>0.
\]
Thus, $S$ is positive definite. Let
\[
T=
\begin{bmatrix}
I & -A^{-1}B^\top\\
\boldsymbol{0} & I
\end{bmatrix}.
\]
Then
\[
T^\top \mathcal{K} T
=
\begin{bmatrix}
A & \boldsymbol{0}\\
\boldsymbol{0} & -S
\end{bmatrix}.
\]
Since $T$ is invertible, $\mathcal{K}$ and $T^\top \mathcal{K} T$ are congruent. By Sylvester's law of inertia, they have the same number of positive and negative eigenvalues. Then, as $A$ and $S$ are positive definite,  it follows that $\mathcal{K}$ has $R$ positive eigenvalues and $M$ negative eigenvalues.
\end{proof}

\begin{remark}
Under the rank assumptions~\eqref{eq:rank_condition_section} and~\eqref{rankM}, 
the constrained least squares problem~\eqref{eq:matrix_problem_sphere} has a unique solution. 
Consequently, the approximant $\hat p_r[f]$ is uniquely defined as an element of $\Pi_r\left(\mathbb{S}^2\right)$.
\end{remark}

We define
\begin{equation}\label{innprod}
    \langle p,q\rangle_{X_N}
:=
\sum_{\boldsymbol{x}_i\in X_N} p\left(\boldsymbol{x}_i\right)q\left(\boldsymbol{x}_i\right).
\end{equation}
By~\eqref{eq:rank_condition_section}, this bilinear form is a discrete inner product on $\Pi_r\left(\mathbb{S}^2\right)$. For any function $h:X_N\to\mathbb{R}$, we denote by 
\begin{equation}\label{normh}
    \|h\|_{2,X_N}
:=
\left(
\sum_{\boldsymbol{x}_i\in X_N}\left|h(\boldsymbol{x}_i)\right|^2
\right)^{1/2},
\end{equation}
the norm induced by the discrete inner product~\eqref{innprod}.
We introduce the spaces
\begin{equation}\label{VrM}
    \mathcal{V}_{r,M}
:=
\left\{
q\in \Pi_r\left(\mathbb{S}^2\right)\, :\, q\left(\boldsymbol{x}_j'\right)=0,\ j=1,\dots,M
\right\},
\end{equation}
and
\begin{equation}\label{ArM}
    \mathcal{A}_{r,M}(f)
:=
\left\{
p\in \Pi_r\left(\mathbb{S}^2\right)\, :\, p\left(\boldsymbol{x}_j'\right)=f\left(\boldsymbol{x}_j'\right),\ j=1,\dots,M
\right\}.
\end{equation}
In the following theorem, we prove that the interpolation--regression approximant $\hat{p}_r[f]$ is the unique element of $\mathcal{A}_{r,M}(f)$ whose residual is orthogonal to $\mathcal{V}_{r,M}$ with respect to $\langle\cdot,\cdot\rangle_{X_N}$.

\begin{theorem}\label{thmorth}
The following statements hold
\begin{itemize}
    \item[(i)] The residual $\hat{p}_r[f]-f$ is orthogonal to $\mathcal{V}_{r,M}$ with respect to
    $\langle\cdot,\cdot\rangle_{X_N}$, namely,
    \[
    \langle \hat{p}_r[f]-f,q\rangle_{X_N}=0,
    \qquad \forall q\in \mathcal{V}_{r,M}.
    \]

    \item[(ii)] $\hat{p}_r[f]$ is the unique element of $\mathcal{A}_{r,M}(f)$ satisfying
    \[
    \langle p-f,q\rangle_{X_N}=0,
    \qquad \forall q\in \mathcal{V}_{r,M}, \quad p\in \mathcal{A}_{r,M}(f).
    \]
\end{itemize}
\end{theorem}

\begin{proof}
We first show that $\mathcal{A}_{r,M}(f)\neq\emptyset$. By~\eqref{rankM}, the linear map
\[
p\in\Pi_r\left(\mathbb{S}^2\right) \mapsto
\left[
p\left(\boldsymbol{x}_1'\right),\dots,p\left(\boldsymbol{x}_M'\right)
\right]^{\top}\in\mathbb{R}^M
\]
is surjective. Therefore, for the prescribed data
\[
\left[
f\left(\boldsymbol{x}_1'\right),\dots,f\left(\boldsymbol{x}_M'\right)
\right]^{\top}\in\mathbb{R}^M,
\]
there exists at least one polynomial $p\in \Pi_r\left(\mathbb{S}^2\right)$ such that
\[
p\left(\boldsymbol{x}_j'\right)=f\left(\boldsymbol{x}_j'\right),
\qquad j=1,\dots,M.
\]
Hence $\mathcal{A}_{r,M}(f)\neq\emptyset$.

Let $q\in\mathcal{V}_{r,M}$. Since $q\left(\boldsymbol{x}_j'\right)=0$ for $j=1,\dots,M$, and $\hat{p}_r[f]\in\mathcal{A}_{r,M}\left(f\right)$, we have
\[
\left(\hat{p}_r[f]+tq\right)\left(\boldsymbol{x}_j'\right)=\hat{p}_r[f]\left(\boldsymbol{x}_j'\right)+tq\left(\boldsymbol{x}_j'\right)=\hat{p}_r[f]\left(\boldsymbol{x}_j'\right)=f\left(\boldsymbol{x}_j'\right), \quad j=1,\dots,M, \quad \forall t\in\mathbb{R},
\]
which shows that
\[
\hat{p}_r[f]+tq\in\mathcal{A}_{r,M}\left(f\right),
\qquad \forall t\in\mathbb{R}.
\]
Define
\[
\phi\left(t\right):=
\sum_{i=1}^N
\left(
\hat{p}_r[f]\left(\boldsymbol{x}_i\right)+tq\left(\boldsymbol{x}_i\right)-f\left(\boldsymbol{x}_i\right)
\right)^2.
\]
Since $\hat{p}_r[f]$ minimizes the constrained least squares functional over $\mathcal{A}_{r,M}(f)$, the function $\phi$ attains its minimum at $t=0$. Hence
\[
\phi'\left(0\right)=0.
\]
Differentiating, we obtain
\[
\phi'\left(t\right)
=
2\sum_{i=1}^N
\left(
\hat{p}_r[f]\left(\boldsymbol{x}_i\right)+tq\left(\boldsymbol{x}_i\right)-f\left(\boldsymbol{x}_i\right)
\right)
q\left(\boldsymbol{x}_i\right).
\]
Evaluating at $t=0$, we have
\[
0=\phi'\left(0\right)
=
2\sum_{i=1}^N
\left(
\hat{p}_r[f]\left(\boldsymbol{x}_i\right)-f\left(\boldsymbol{x}_i\right)
\right)
q\left(\boldsymbol{x}_i\right)
=
2\left\langle \hat{p}_r[f]-f,q\right\rangle_{X_N}.
\]
Therefore
\[
\left\langle \hat{p}_r[f]-f,q\right\rangle_{X_N}=0,
\qquad \forall q\in\mathcal{V}_{r,M}.
\]
This proves $(i)$.
\medskip

 Let $p\in\mathcal{A}_{r,M}\left(f\right)$ satisfy
\begin{equation}\label{ortcond}
 \left\langle p-f,q\right\rangle_{X_N}=0,
\qquad \forall q\in\mathcal{V}_{r,M}.   
\end{equation}
Since both $p$ and $\hat{p}_r[f]$ belong to $\mathcal{A}_{r,M}\left(f\right)$, they satisfy the same interpolation conditions on $X_M$, and therefore
\[
\left(p-\hat{p}_r[f]\right)\left(\boldsymbol{x}_j'\right)=0,
\qquad j=1,\dots,M,
\]
that is
\[
p-\hat{p}_r[f]\in\mathcal{V}_{r,M}.
\]
Taking $q=p-\hat{p}_r[f]$ in the orthogonality relations~\eqref{ortcond}, we obtain
\[
\left\langle p-f,p-\hat{p}_r[f]\right\rangle_{X_N}=0,
\]
and by $(i)$, we also have
\[
\left\langle \hat{p}_r[f]-f,p-\hat{p}_r[f]\right\rangle_{X_N}=0.
\]
Subtracting the two identities yields
\[
\left\langle p-\hat{p}_r[f],p-\hat{p}_r[f]\right\rangle_{X_N}=0.
\]
Hence $p=\hat{p}_r[f]$, which proves uniqueness. This concludes the proof.
\end{proof}

\begin{remark}
\label{rem:polynomial_reproduction}
If $f \in \Pi_r\left(\mathbb{S}^2\right)$, then $\hat p_r[f]=f$. 
Indeed, $f \in \mathcal{A}_{r,M}(f)$ and the discrete least squares functional vanishes at $f$, so the claim follows from uniqueness of the minimizer.
\end{remark}

\begin{remark}
The previous result is purely algebraic and does not depend on the geometry of the sphere. In particular, both the solvability of the KKT system and the orthogonal characterization of the interpolation--regression approximant depend only on the rank conditions of the associated Vandermonde matrices.
\end{remark}

\subsection{Antipodal symmetry and parity decomposition}

We now consider the antipodally symmetric setting, where the parity of the spherical harmonics plays an important role.  Throughout this subsection, we assume that the sampling set $X_N$ and the interpolation set $X_M$ are antipodally symmetric, that is
\[
\boldsymbol{x}_i\in X_N \implies -\boldsymbol{x}_i\in X_N,
\qquad
\boldsymbol{x}_j^{\prime}\in X_M \implies -\boldsymbol{x}_j^{\prime}\in X_M.
\]
We introduce the subspaces
\[
\Pi_r^{+}\left(\mathbb{S}^2\right)
:=
\left\{
p\in \Pi_r\left(\mathbb{S}^2\right) \,:\, p\left(-\boldsymbol{x}\right)=p\left(\boldsymbol{x}\right)
\ \text{for all }\boldsymbol{x}\in\mathbb{S}^2
\right\},
\]
and
\[
\Pi_r^{-}\left(\mathbb{S}^2\right)
:=
\left\{
p\in \Pi_r\left(\mathbb{S}^2\right)\,:\, p\left(-\boldsymbol{x}\right)=-p\left(\boldsymbol{x}\right)
\ \text{for all }\boldsymbol{x}\in\mathbb{S}^2
\right\}.
\]
Under this assumption, the space $\Pi_r\left(\mathbb{S}^2\right)$ admits a natural decomposition into even and odd components.

\begin{proposition}\label{prop1}
The subspaces $\Pi_r^{+}\left(\mathbb{S}^2\right)$ and $\Pi_r^{-}\left(\mathbb{S}^2\right)$ are orthogonal with respect to the discrete inner product~\eqref{innprod}. Moreover, we have
\begin{equation}\label{decomp}
\Pi_r\left(\mathbb{S}^2\right)=\Pi_r^{+}\left(\mathbb{S}^2\right)\oplus \Pi_r^{-}\left(\mathbb{S}^2\right).
\end{equation}
\end{proposition}

\begin{proof}
Let $p\in \Pi_r^{+}\left(\mathbb{S}^2\right)$ and $q\in \Pi_r^{-}\left(\mathbb{S}^2\right)$. Since $X_N$ is antipodally symmetric, its elements can be partitioned into antipodal pairs $\left\{\boldsymbol{x}_i,-\boldsymbol{x}_i\right\}$. For each such pair, we have
\[
p\left(-\boldsymbol{x}_i\right)q\left(-\boldsymbol{x}_i\right)
=
p\left(\boldsymbol{x}_i\right)\left(-q\left(\boldsymbol{x}_i\right)\right)
=
-p\left(\boldsymbol{x}_i\right)q\left(\boldsymbol{x}_i\right).
\]
Hence the contributions of each antipodal pair is zero. Therefore
\[
\langle p,q\rangle_{X_N}=0.
\]

For any $p\in \Pi_r\left(\mathbb{S}^2\right)$, the decomposition~\eqref{decomp} follows from the standard splitting
\[
p^{+}\left(\boldsymbol{x}\right):=\frac{p\left(\boldsymbol{x}\right)+p\left(-\boldsymbol{x}\right)}{2},
\qquad
p^{-}\left(\boldsymbol{x}\right):=\frac{p\left(\boldsymbol{x}\right)-p\left(-\boldsymbol{x}\right)}{2}.
\]
Then
\[
p=p^{+}+p^{-}.
\]
\end{proof}

\begin{remark}
    In the antipodally symmetric setting, the cardinality $M = (m+1)^2$ of the interpolation set must be even. This implies that $m$ must be odd. Therefore, the convention $M = (m+1)^2$ is compatible with antipodal symmetry only for odd values of $m$.
\end{remark}

\begin{remark}
Using the identities
\begin{eqnarray*}
P_\ell^m(-t)&=&(-1)^{\ell+m}P_\ell^m(t),\\
\cos\left(m(\theta+\pi)\right)&=&(-1)^m\cos(m\theta),\\
\sin\left(m(\theta+\pi)\right)&=&(-1)^m\sin(m\theta),
\end{eqnarray*}
together with the fact that the antipodal point of $\boldsymbol{x}(\varphi,\theta)$ is
$\boldsymbol{x}(\pi-\varphi,\theta+\pi)$, it follows that the real spherical harmonics satisfy
\begin{equation}\label{parcond}
    u_{\ell,m}\left(-\boldsymbol{x}\right)=(-1)^\ell u_{\ell,m}\left(\boldsymbol{x}\right),
\qquad \boldsymbol{x}\in\mathbb{S}^2.
\end{equation}
\end{remark}

The antipodal symmetry of the sampling and interpolation sets, together with the
parity of the real spherical harmonics, allows the constrained problem to be
decomposed into even and odd components. Let $\mathcal{B}_r$ be the basis defined
in~\eqref{real_armsfbas}, and set
\[
\mathcal{B}_r^{+}
:=
\left\{
u_{\ell,m}\in\mathcal{B}_r \,:\, \ell \ \text{even}
\right\},
\qquad
\mathcal{B}_r^{-}
:=
\left\{
u_{\ell,m}\in\mathcal{B}_r \,:\, \ell \ \text{odd}
\right\}.
\]
In what follows, the spherical harmonics are ordered by listing first the elements of
$\mathcal{B}_r^{+}$ and then those of $\mathcal{B}_r^{-}$.
For any function $f:\mathbb{S}^2\to\mathbb{R}$, define its even and odd parts by
\[
f^{+}(\boldsymbol{x})
:=
\frac{f(\boldsymbol{x})+f(-\boldsymbol{x})}{2},
\qquad
f^{-}(\boldsymbol{x})
:=
\frac{f(\boldsymbol{x})-f(-\boldsymbol{x})}{2},
\qquad \boldsymbol{x}\in\mathbb{S}^2.
\]

\begin{theorem}\label{thm:parity_decoupling}
For any function $f:\mathbb{S}^2\to\mathbb{R}$ the interpolation--regression approximant $\hat p_r[f]$ admits the representation
\begin{equation}\label{eq:parity_split_approx}
\hat p_r[f]
=
\hat p_r^{+}\left[f^{+}\right]
+
\hat p_r^{-}\left[f^{-}\right],
\end{equation}
where $\hat p_r^{+}\left[f^{+}\right]\in \Pi_r^{+}\left(\mathbb{S}^2\right)$ and $\hat p_r^{-}\left[f^{-}\right]\in \Pi_r^{-}\left(\mathbb{S}^2\right)$ are uniquely determined by the two independent constrained minimization problems
\begin{equation}\label{eq:even_problem}
\hat p_r^{+}\left[f^{+}\right]
=
\operatorname*{argmin}_{p\in \Pi_r^{+}\left(\mathbb{S}^2\right)}
\sum_{\boldsymbol{x}_i\in X_N}
\left|
p\left(\boldsymbol{x}_i\right)-f^{+}\left(\boldsymbol{x}_i\right)
\right|^2
\quad
\text{subject to}
\quad
p\left(\boldsymbol{x}_j^{\prime}\right)=f^{+}\left(\boldsymbol{x}_j^{\prime}\right),
\ \boldsymbol{x}_j^{\prime}\in X_M,
\end{equation}
and
\begin{equation}\label{eq:odd_problem}
\hat p_r^{-}\left[f^{-}\right]
=
\operatorname*{argmin}_{p\in \Pi_r^{-}\left(\mathbb{S}^2\right)}
\sum_{\boldsymbol{x}_i\in X_N}
\left|
p\left(\boldsymbol{x}_i\right)-f^{-}\left(\boldsymbol{x}_i\right)
\right|^2
\quad
\text{subject to}
\quad
p\left(\boldsymbol{x}_j^{\prime}\right)=f^{-}\left(\boldsymbol{x}_j^{\prime}\right),
\ \boldsymbol{x}_j^{\prime}\in X_M.
\end{equation}
\end{theorem}
\begin{proof}
For any $p\in\Pi_r\left(\mathbb{S}^2\right)$, we write
\[
p=p^++p^-, \quad p^{+}\in\Pi_r^{+}\left(\mathbb{S}^2\right), \quad p^-\in\Pi_r^{-}\left(\mathbb{S}^2\right).
\]
Since $X_M$ is antipodally symmetric, the interpolation conditions
\[
p\left(\boldsymbol{x}_j^{\prime}\right)=f\left(\boldsymbol{x}_j^{\prime}\right),
\qquad \boldsymbol{x}_j^{\prime}\in X_M,
\]
are equivalent, by adding and subtracting the equations on each antipodal pair, to
\begin{equation}\label{seceqa}
p^{+}\left(\boldsymbol{x}_j^{\prime}\right)=f^{+}\left(\boldsymbol{x}_j^{\prime}\right),
\qquad
p^{-}\left(\boldsymbol{x}_j^{\prime}\right)=f^{-}\left(\boldsymbol{x}_j^{\prime}\right),
\qquad
\boldsymbol{x}_j^{\prime}\in X_M.
\end{equation}

Let
\begin{equation}\label{fun_J}
    \mathcal{J}(p;f)
:=
\sum_{\boldsymbol{x}_i\in X_N}
\left| p\left(\boldsymbol{x}_i\right)-f\left(\boldsymbol{x}_i\right) \right|^2.
\end{equation}
Using the decomposition
\[
p-f=\left(p^{+}-f^{+}\right)+\left(p^{-}-f^{-}\right),
\]
we have
\begin{eqnarray*}
\mathcal{J}(p;f)
&=&
\sum_{\boldsymbol{x}_i\in X_N}
\left|
\left(p^{+}-f^{+}\right)\left(\boldsymbol{x}_i\right)
+
\left(p^{-}-f^{-}\right)\left(\boldsymbol{x}_i\right)
\right|^2 \\
&=&
\sum_{\boldsymbol{x}_i\in X_N}
\left| \left(p^{+}-f^{+}\right)\left(\boldsymbol{x}_i\right) \right|^2
+
\sum_{\boldsymbol{x}_i\in X_N}
\left| \left(p^{-}-f^{-}\right)\left(\boldsymbol{x}_i\right) \right|^2 
+
2\left\langle p^{+}-f^{+} , p^{-}-f^{-}\right\rangle_{X_N}.
\end{eqnarray*}
Since $p^{+}-f^{+}$ is even and $p^{-}-f^{-}$ is odd, and since $X_N$ is antipodally symmetric, we have
\[
\left\langle p^{+}-f^{+} , p^{-}-f^{-}\right\rangle_{X_N}=0.
\]
Hence
\begin{equation}\label{eq:energy_split}
\mathcal{J}(p;f)
=
\sum_{\boldsymbol{x}_i\in X_N}
\left| p^{+}\left(\boldsymbol{x}_i\right)-f^{+}\left(\boldsymbol{x}_i\right) \right|^2
+
\sum_{\boldsymbol{x}_i\in X_N}
\left| p^{-}\left(\boldsymbol{x}_i\right)-f^{-}\left(\boldsymbol{x}_i\right) \right|^2.
\end{equation}
The original constrained minimization problem consists of minimizing $\mathcal{J}(p;f)$ over all
\[
p\in \Pi_r\left(\mathbb{S}^2\right)
\]
subject to the interpolation conditions
\[
p\left(\boldsymbol{x}_j^{\prime}\right)=f\left(\boldsymbol{x}_j^{\prime}\right),
\qquad
\boldsymbol{x}_j^{\prime}\in X_M.
\]
Therefore, minimizing the original constrained functional is equivalent to minimizing independently
the even and odd contributions in~\eqref{eq:energy_split} under the corresponding  constraints~\eqref{seceqa}. Hence the minimizer $\hat p_r[f]$ decomposes as
\[
\hat p_r[f]
=
\hat p_r^{+}\left[f^{+}\right]
+
\hat p_r^{-}\left[f^{-}\right],
\]
where $\hat p_r^{+}\left[f^{+}\right]$ and $\hat p_r^{-}\left[f^{-}\right]$ solve
\eqref{eq:even_problem} and~\eqref{eq:odd_problem}, respectively. The uniqueness follows from the uniqueness of the original constrained least squares problem~\eqref{eq:matrix_problem_sphere}.
\end{proof}

\begin{remark}
Since $X_M$ is antipodally symmetric, it is enough in
\eqref{eq:even_problem} and~\eqref{eq:odd_problem} to impose the interpolation
conditions on one representative of each antipodal pair.
\end{remark}

Antipodal symmetry also implies a block structure for the discrete normal matrix  associated with the spherical Vandermonde matrix. In particular, the conditioning of the full Vandermonde matrix is determined by its even and odd components.

\begin{theorem}\label{thm:spectral_split}
The spherical Vandermonde matrix can be written as
\[
V_{N,r} = \left[ V_{N,r}^{+} \;\; V_{N,r}^{-} \right],
\]
where $V_{N,r}^{+}$ and $V_{N,r}^{-}$ denote the submatrices corresponding to the even and odd basis functions, respectively. The following relations hold
\begin{align}
\sigma_{\max}(V_{N,r}) &= \max \left\{ \sigma_{\max}(V_{N,r}^{+}), \sigma_{\max}(V_{N,r}^{-}) \right\}, \label{eq:smax_split} \\
\sigma_{\min}(V_{N,r}) &= \min \left\{ \sigma_{\min}(V_{N,r}^{+}), \sigma_{\min}(V_{N,r}^{-}) \right\}, \label{eq:smin_split}
\end{align}
and consequently
\begin{equation}\label{eq:cond_split}
\kappa_2(V_{N,r}) = \frac{
\max \left\{ \sigma_{\max}(V_{N,r}^{+}), \sigma_{\max}(V_{N,r}^{-}) \right\}
}{
\min \left\{ \sigma_{\min}(V_{N,r}^{+}), \sigma_{\min}(V_{N,r}^{-}) \right\}
}.
\end{equation}
\end{theorem}

\begin{proof}
Since $X_N$ is antipodally symmetric, we have
\[
\left(V_{N,r}^{+}\right)^{\top}V_{N,r}^{-}=0.
\]
Therefore, the discrete normal matrix admits the block diagonal decomposition
\begin{equation}\label{eq:block_normal_matrix}
V_{N,r}^{\top}V_{N,r}
=
\begin{bmatrix}
\left(V_{N,r}^{+}\right)^{\top}V_{N,r}^{+} & 0 \\[4pt]
0 & \left(V_{N,r}^{-}\right)^{\top}V_{N,r}^{-}
\end{bmatrix}.
\end{equation}
Since $\operatorname{rank}\left(V_{N,r}\right)=R$, the columns of $V_{N,r}$ are linearly independent, and therefore any subset of columns is linearly independent. The eigenvalues of a block diagonal matrix are precisely the eigenvalues of its diagonal blocks, counted with multiplicity. It follows that the singular values of $V_{N,r}$ are obtained by collecting the singular values of $V^+_{N,r}$ and $V^-_{N,r}$.  The rest of the proof follows immediately from this observation.
\end{proof}

The parity structure of the real spherical harmonics implies an equivariance property of the interpolation--regression approximant under the antipodal map. This is formalized in the following theorem.

\begin{theorem}\label{thm:antipodal-equivariance}
For any function $f:\mathbb{S}^2\to\mathbb{R}$ 
\begin{equation}\label{eq:equivariance_antipodal}
\hat{p}_r\left[f^\star\right]=\left(\hat{p}_r[f]\right)^\star ,
\end{equation}
where
\[
f^\star(\boldsymbol{x}):=f(-\boldsymbol{x}),
\qquad \boldsymbol{x}\in\mathbb{S}^2.
\]
Thus, the interpolation--regression operator on the sphere is equivariant under the antipodal map.
\end{theorem}

\begin{proof}
Let
\[
\hat{p}_r[f]=\sum_{\ell=0}^r\sum_{m=-\ell}^{\ell} c_{\ell,m}u_{\ell,m}, \quad c_{\ell,m}\in\mathbb{R}.
\]
By the parity relation~\eqref{parcond} 
we obtain
\[
\left(\hat{p}_r[f]\right)^\star(\boldsymbol{x})
=
\hat{p}_r[f](-\boldsymbol{x})
=
\sum_{\ell=0}^r\sum_{m=-\ell}^{\ell} (-1)^\ell c_{\ell,m}u_{\ell,m}(\boldsymbol{x})\in\Pi_r\left(\mathbb{S}^2\right).
\]
We show that $\left(\hat{p}_r[f]\right)^\star$ is the interpolation--regression approximant associated with $f^\star$. Since $\hat{p}_r[f]$ satisfies the interpolation conditions on $X_M$, we have
\[
\hat{p}_r[f]\left(\boldsymbol{x}_j^{\prime}\right)=f\left(\boldsymbol{x}_j^{\prime}\right), \quad \boldsymbol{x}_j^{\prime}\in X_M.
\]
As $X_M$ is antipodally symmetric, we get
\[
\left(\hat{p}_r[f]\right)^\star\left(\boldsymbol{x}_j^{\prime}\right)
=
\hat{p}_r[f]\left(-\boldsymbol{x}_j^{\prime}\right)
=
f\left(-\boldsymbol{x}_j^{\prime}\right)
=
f^\star\left(\boldsymbol{x}_j^{\prime}\right).
\]
Thus $\left(\hat{p}_r[f]\right)^\star$ satisfies the interpolation constraints associated with $f^\star$.

Now let $q\in\Pi_r\left(\mathbb{S}^2\right)$ be any polynomial such that
\[
q\left(\boldsymbol{x}_j^{\prime}\right)=f^\star\left(\boldsymbol{x}_j^{\prime}\right),
\qquad \boldsymbol{x}_j^{\prime}\in X_M.
\]
Since $X_M$ is antipodally symmetric, it follows that
\[q\left(-\boldsymbol{x}_j^{\prime}\right)=f^\star\left(-\boldsymbol{x}_j^{\prime}\right),
\qquad \boldsymbol{x}_j^{\prime}\in X_M,\]
and then
\[
q^\star\left(\boldsymbol{x}_j^{\prime}\right)=f\left(\boldsymbol{x}_j^{\prime}\right),
\qquad \boldsymbol{x}_j^{\prime}\in X_M.
\]
By the minimality of $\hat{p}_r[f]$, we obtain
\[
\sum_{\boldsymbol{x}_i\in X_N}
\left|
\hat{p}_r[f]\left(\boldsymbol{x}_i\right)-f\left(\boldsymbol{x}_i\right)
\right|^2
\le
\sum_{\boldsymbol{x}_i\in X_N}
\left|
q^\star\left(\boldsymbol{x}_i\right)-f\left(\boldsymbol{x}_i\right)
\right|^2.
\]
Using the antipodal symmetry of $X_N$, we can rewrite the inequality as
\[
\sum_{\boldsymbol{x}_i\in X_N}
\left|
\left(\hat{p}_r[f]\right)^\star\left(\boldsymbol{x}_i\right)-f^\star\left(\boldsymbol{x}_i\right)
\right|^2
\le
\sum_{\boldsymbol{x}_i\in X_N}
\left|
q\left(\boldsymbol{x}_i\right)-f^\star\left(\boldsymbol{x}_i\right)
\right|^2.
\]
Therefore $\left(\hat{p}_r[f]\right)^\star$ is a minimizer of the constrained least squares problem associated with $f^\star$. By uniqueness, we conclude that
\[
\hat{p}_r\left[f^\star\right]=\left(\hat{p}_r[f]\right)^\star.
\]
This proves~\eqref{eq:equivariance_antipodal}.
\end{proof}

\begin{corollary}
The following conditions hold
\begin{itemize}
    \item[(i)] If $f^\star=-f$ on $X_N$, then $\hat{p}_r[f]$ is odd. 
    \item[(ii)] If $f^\star=f$ on $X_N$, then $\hat{p}_r[f]$ is even.
\end{itemize}
\end{corollary}

\begin{proof}
We prove only $(i)$, since $(ii)$ follows analogously.

Assume that $f^\star=-f$ on $X_N$. We first observe that the polynomial $-\hat{p}_r[f]$ satisfies the interpolation conditions associated with $-f$ on $X_M$. Indeed
\[
-\hat{p}_r[f]\left(\boldsymbol{x}_j^{\prime}\right) = -f\left(\boldsymbol{x}_j^{\prime}\right), \quad \boldsymbol{x}_j^{\prime}\in X_M.
\]
Moreover, we have
\[
\sum_{\boldsymbol{x}_i\in X_N}
\left|
-\hat{p}_r[f]\left(\boldsymbol{x}_i\right)-\left(-f\left(\boldsymbol{x}_i\right)\right)
\right|^2
=
\sum_{\boldsymbol{x}_i\in X_N}
\left|
\hat{p}_r[f]\left(\boldsymbol{x}_i\right)-f\left(\boldsymbol{x}_i\right)
\right|^2.
\]
Therefore, $-\hat{p}_r[f]$ is a minimizer for the interpolation--regression problem associated with $-f$. By uniqueness of the approximant, we conclude that
\begin{equation}\label{condparf}
    \hat{p}_r[-f]=-\hat{p}_r[f].
\end{equation}
Since $f^\star=-f$ on $X_N$, it follows that
\[
\hat{p}_r\left[f^\star\right]=\hat{p}_r[-f].
\]
Then, combining~\eqref{eq:equivariance_antipodal} and~\eqref{condparf}, 
we obtain
\[
\left(\hat{p}_r[f]\right)^\star
=
\hat{p}_r\left[f^\star\right]
=
\hat{p}_r[-f]
=
-\hat{p}_r[f].
\]
Hence $\hat{p}_r[f]$ is odd.
\end{proof}

The antipodal symmetry considered above is a particular instance of a more general equivariance property. More generally, the interpolation--regression operator is equivariant under orthogonal transformations of the sphere that preserve both the sampling set and the interpolation set. For any $Q\in\mathbb{R}^{3\times 3}$, we denote by
\begin{eqnarray*}
    QX_N&:=&\left\{Q\boldsymbol{x}_i\,:\, i=1,\dots,N\right\},\\
    QX_M&:=&\left\{Q\boldsymbol{x}_j^{\prime}\,:\, j=1,\dots,M\right\}.
\end{eqnarray*}

\begin{theorem}\label{thm:group-equivariance}
Let $Q\in\mathbb{R}^{3\times 3}$ be an orthogonal matrix such that
\[
QX_N = X_N, \quad QX_M = X_M.
\]
For any function $f:\mathbb{S}^2\to\mathbb{R}$, define
\begin{equation}
    \label{fQepQ}
    f^{Q}:=f\circ Q^{-1}.
\end{equation}
Then
\begin{equation}\label{eq:equivariance_group}
\hat{p}_r\left[f^{Q}\right]=\left(\hat{p}_r[f]\right)^{Q}.
\end{equation}
In other words, the interpolation--regression operator on the sphere is equivariant with respect to every orthogonal symmetry preserving both $X_N$ and $X_M$.
\end{theorem}

\begin{proof}
Since $Q$ is orthogonal, it preserves the Euclidean norm, and therefore maps $\mathbb{S}^2$ onto itself. In particular, if $p\in\Pi_r\left(\mathbb{S}^2\right)$, then
\[
p^{Q}=p\circ Q^{-1}=p\circ Q^{\top}
\]
still belongs to $\Pi_r\left(\mathbb{S}^2\right)$.  
We show that $\left(\hat{p}_r[f]\right)^{Q}$ coincides with the interpolation--regression approximant associated with $f^{Q}$. Since $\hat{p}_r[f]$ satisfies the interpolation conditions on $X_M$, for every $\boldsymbol{x}_j^{\prime}\in X_M$ we have
\[
\hat{p}_r[f]\left(\boldsymbol{x}_j^{\prime}\right)=f\left(\boldsymbol{x}_j^{\prime}\right).
\]
Since $QX_M=X_M$, it follows that $Q^\top\boldsymbol{x}_j^{\prime}\in X_M$, and therefore
\[
\left(\hat{p}_r[f]\right)^{Q}\left(\boldsymbol{x}_j^{\prime}\right)
=
\hat{p}_r[f]\left(Q^\top\boldsymbol{x}_j^{\prime}\right)
=
f\left(Q^\top\boldsymbol{x}_j^{\prime}\right)
=
f^{Q}\left(\boldsymbol{x}_j^{\prime}\right).
\]

Let $p\in\Pi_r\left(\mathbb{S}^2\right)$ be any polynomial such that
\begin{equation}\label{dcsss}
p\left(\boldsymbol{x}_j^{\prime}\right)=f^{Q}\left(\boldsymbol{x}_j^{\prime}\right),
\qquad \boldsymbol{x}_j^{\prime}\in X_M.
\end{equation}
Hence, we have
\[
(p\circ Q)\left(\boldsymbol{x}_j^{\prime}\right)
=
p\left(Q\boldsymbol{x}_j^{\prime}\right)
=
f^{Q}\left(Q\boldsymbol{x}_j^{\prime}\right)
=
f\left(\boldsymbol{x}_j^{\prime}\right).
\]
Thus $p\circ Q$ satisfies the interpolation constraints associated with $f$. By the minimality of $\hat{p}_r[f]$, we obtain
\begin{equation}\label{new1}
    \sum_{\boldsymbol{x}_i\in X_N}
\left|
\hat{p}_r[f]\left(\boldsymbol{x}_i\right)-f\left(\boldsymbol{x}_i\right)
\right|^2
\le
\sum_{\boldsymbol{x}_i\in X_N}
\left|
(p\circ Q)\left(\boldsymbol{x}_i\right)-f\left(\boldsymbol{x}_i\right)
\right|^2.
\end{equation}
Since $QX_N=X_N$, we can reindex the discrete sums by setting $\boldsymbol{y}_i=Q\boldsymbol{x}_i\in X_N$, and obtain
\begin{eqnarray}\notag
    \sum_{\boldsymbol{x}_i\in X_N}
\left|
\hat{p}_r[f]\left(\boldsymbol{x}_i\right)-f\left(\boldsymbol{x}_i\right)
\right|^2
&=&
\sum_{\boldsymbol{y}_i\in X_N}
\left|
\hat{p}_r[f]\left(Q^{-1}\boldsymbol{y}_i\right)-f\left(Q^{-1}\boldsymbol{y}_i\right)
\right|^2\\&=& \sum_{\boldsymbol{y}_i\in X_N}
\left|
\left(\hat{p}_r[f]\right)^{Q}\left(\boldsymbol{y}_i\right)-f^{Q}\left(\boldsymbol{y}_i\right)
\right|^2.\label{rfev}
\end{eqnarray} 
Similarly we have
\begin{equation}\label{rfev1}
    \sum_{\boldsymbol{x}_i\in X_N}
\left|
(p\circ Q)\left(\boldsymbol{x}_i\right)-f\left(\boldsymbol{x}_i\right)
\right|^2
=
\sum_{\boldsymbol{y}_i\in X_N}
\left|
p\left(\boldsymbol{y}_i\right)-f^{Q}\left(\boldsymbol{y}_i\right)
\right|^2.
\end{equation}
Therefore, by substituting~\eqref{rfev} and~\eqref{rfev1} into~\eqref{new1}, we have
\[
\sum_{\boldsymbol{x}_i\in X_N}
\left|
\left(\hat{p}_r[f]\right)^{Q}\left(\boldsymbol{x}_i\right)-f^{Q}\left(\boldsymbol{x}_i\right)
\right|^2
\le
\sum_{\boldsymbol{x}_i\in X_N}
\left|
p\left(\boldsymbol{x}_i\right)-f^{Q}\left(\boldsymbol{x}_i\right)
\right|^2.
\]
Hence $\left(\hat{p}_r[f]\right)^{Q}$ is a minimizer of the constrained least squares problem associated with $f^{Q}$. By uniqueness of the interpolation--regression approximant, we conclude that
\[
\hat{p}_r\left[f^{Q}\right]=\left(\hat{p}_r[f]\right)^{Q}.
\]
\end{proof}

\begin{remark}
In Theorem~\ref{thm:group-equivariance}, we show that the interpolation--regression operator is equivariant with respect to orthogonal transformations preserving the node sets. This property reflects the fact that the approximation process is compatible with the symmetry of the underlying discrete configuration, and therefore preserves the geometric invariance of the problem.
\end{remark}

\section{The case of spherical designs}
\label{sec:designs}

In this section, we consider the case in which the sampling set satisfies an exact quadrature property on the sphere.

A finite set $X_N=\left\{\boldsymbol{x}_1,\dots,\boldsymbol{x}_N\right\}\subset \mathbb{S}^2$ is called a \textit{spherical $k$-design}~\cite{bajnok1992construction,zhou2018spherical} if
\[
\frac{1}{N}\sum_{i=1}^N p(\boldsymbol{x}_i)
=
\frac{1}{4\pi}\int_{\mathbb{S}^2} p(\boldsymbol{x}) d\omega(\boldsymbol{x}),
\qquad
\forall p\in \Pi_k\left(\mathbb{S}^2\right).
\]
Throughout this section, we assume that $X_N \subset \mathbb{S}^2$ is a spherical $k$-design with $k \ge 2r$.

The next result will be useful in the following analysis.
\begin{theorem}\label{thm:design_orthogonality}
For any
$p,q \in \Pi_r\left(\mathbb{S}^2\right)$, we have
\begin{equation}
\frac{1}{N}\sum_{\boldsymbol{x}_i\in X_N} p\left(\boldsymbol{x}_i\right)q\left(\boldsymbol{x}_i\right)
=
\frac{1}{4\pi}\int_{\mathbb{S}^2} p(\boldsymbol{x})q(\boldsymbol{x})d\omega(\boldsymbol{x}).
\label{eq:design-pq}
\end{equation}
Equivalently
\begin{equation}
\langle p,q\rangle_{X_N}
=
\frac{N}{4\pi}\langle p,q\rangle_{L^2\left(\mathbb{S}^2\right)}.
\label{eq:inner-product-design}
\end{equation}
\end{theorem}

\begin{proof}
Let $p,q\in \Pi_r\left(\mathbb{S}^2\right)$. Then
\[
pq \in \Pi_{2r}\left(\mathbb{S}^2\right).
\]
Since $k\ge 2r$ and $X_N$ is a spherical $k$-design, the quadrature property applied to $pq$ gives
\[
\sum_{i=1}^N p\left(\boldsymbol{x}_i\right)q\left(\boldsymbol{x}_i\right)
=
\frac{N}{4\pi}\int_{\mathbb{S}^2} p(\boldsymbol{x})q(\boldsymbol{x}) d\omega(\boldsymbol{x}).
\]
The identity~\eqref{eq:inner-product-design} immediately follows from the definition of $\langle\cdot,\cdot\rangle_{X_N}$.
\end{proof}

As an immediate consequence, the discrete Gram matrix admits a particularly simple form.

\begin{corollary}
\label{cor:design_gram}
Let $\mathcal{B}_r$ be the basis of real spherical harmonics defined in~\eqref{real_armsfbas}. Then
\[
V_{N,r}^\top V_{N,r}
=
\frac{N}{4\pi} I.
\]
\end{corollary}

\begin{proof}
The entries of $V_{N,r}^\top V_{N,r}$ are
\begin{equation}\label{fisaap}
    \left[V_{N,r}^\top V_{N,r}\right]_{(\ell,m),(\ell',m')}
=
\sum_{i=1}^N
u_{\ell,m}\left(\boldsymbol{x}_i\right)
u_{\ell',m'}\left(\boldsymbol{x}_i\right)=\langle u_{\ell,m},u_{\ell',m'}\rangle_{X_N}.
\end{equation}
Combining~\eqref{eq:inner-product-design} and~\eqref{fisaap}, we get
\[
\left[V_{N,r}^\top V_{N,r}\right]_{(\ell,m),(\ell',m')}
=
\frac{N}{4\pi}\langle u_{\ell,m},u_{\ell',m'}\rangle_{L^2\left(\mathbb{S}^2\right)}.
\]
Since the functions $u_{\ell,m}$ satisfy~\eqref{ortcond1}, we have
\[
V_{N,r}^\top V_{N,r}
=
\frac{N}{4\pi} I.
\]
\end{proof}

\begin{remark}
In the spherical $k$-design setting all singular values of
$V_{N,r}$ coincide. In particular, $V_{N,r}$ is optimally conditioned in the
spectral norm. Moreover, the KKT matrix can be written as
\begin{equation}\label{eq:KKT_design_form}
\mathcal{K}
=
\begin{bmatrix}
\alpha I & V_{M,r}^\top \\
V_{M,r} & 0
\end{bmatrix},
\qquad
\alpha=\frac{N}{4\pi}.
\end{equation}
The Schur complement of the leading block $\alpha I$ is
\[
-\frac{1}{\alpha}V_{M,r}V_{M,r}^\top.
\]
Equivalently, with the notation used in~\eqref{ShurComp}, it results
\begin{equation}
    \label{eq:Schur_design}
    S=\frac{1}{\alpha}V_{M,r}V_{M,r}^\top.
\end{equation}
\end{remark}

The next result gives an explicit description of the spectrum of the KKT matrix. 

\begin{theorem}\label{thm:KKT_spectrum_design}
Let $X_M \subset X_N$ be an interpolation subset satisfying~\eqref{rankM}, and let
\[
\sigma_1\left(V_{M,r}\right)\geq \sigma_2\left(V_{M,r}\right)\geq \cdots \geq \sigma_M\left(V_{M,r}\right)>0
\]
be the singular values of $V_{M,r}$. Then the following statements hold.

\begin{itemize}
\item[(i)] The spectrum of the KKT matrix
\begin{equation}\label{matKaa}
    \mathcal K=
\begin{bmatrix}
\alpha I & V_{M,r}^\top\\
V_{M,r} & 0
\end{bmatrix}
\end{equation}
is given by
\[
\operatorname{spec}(\mathcal K)
=
\{\alpha\}^{R-M}
\cup
\left\{
\lambda_j^+,\lambda_j^- : j=1,\dots,M
\right\},
\]
where
\begin{equation}\label{eq:KKT_eigs_design}
\lambda_j^\pm
=
\frac{\alpha\pm\sqrt{\alpha^2+4\alpha\mu_j}}{2},
\qquad j=1,\dots,M,
\end{equation}
and $\mu_j$, $j=1,\dots,M$ are the eigenvalues of $S$.
In particular, $\lambda_j^+>0$ and $\lambda_j^-<0$ for any $j=1,\dots,M$.

\item[(ii)] The singular values of $\mathcal K$ are
\[
\{\alpha\}^{R-M}
\cup
\left\{
\frac{\alpha+\sqrt{\alpha^2+4\sigma_j\left(V_{M,r}\right)^2}}{2},
\frac{\sqrt{\alpha^2+4\sigma_j\left(V_{M,r}\right)^2}-\alpha}{2}
:\ j=1,\dots,M
\right\}.
\]

\item[(iii)] The condition number of $\mathcal K$ is
\begin{equation}\label{eq:KKT_cond_design}
\kappa_2(\mathcal K)
=
\frac{
\frac{\alpha+\sqrt{\alpha^2+4\sigma_{\max}\left(V_{M,r}\right)^2}}{2}
}{
\min\left\{
\alpha,
\frac{\sqrt{\alpha^2+4\sigma_{\min}\left(V_{M,r}\right)^2}-\alpha}{2}
\right\}
}.
\end{equation}
\end{itemize}
\end{theorem}
\begin{proof}
Since
\[
\operatorname{rank}\left(V_{M,r}\right)=M,
\]
the matrix
\[
S=\frac{1}{\alpha}V_{M,r}V_{M,r}^\top\in\mathbb{R}^{M\times M}
\]
is symmetric positive definite. Hence, by the spectral theorem, there exists an orthonormal basis 
\[
\left\{\boldsymbol{u}_1,\dots,\boldsymbol{u}_M\right\}
\]
of eigenvectors of $S$, namely
\[
S\boldsymbol{u}_j=\mu_j\boldsymbol{u}_j,
\qquad j=1,\dots,M.
\]
Then, by~\eqref{eq:Schur_design}, we get
\begin{equation}\label{newnews}
    V_{M,r}V_{M,r}^\top \boldsymbol{u}_j=\alpha\mu_j \boldsymbol{u}_j,
\qquad j=1,\dots,M.
\end{equation}
Since $S$ is positive definite, we have
\[
\mu_j>0,
\qquad j=1,\dots,M.
\]
Then, for every $j=1,\dots,M$, we set
\begin{equation}\label{defz}
\boldsymbol{z}_j
=
\frac{1}{\sqrt{\alpha\mu_j}}\,V_{M,r}^\top \boldsymbol{u}_j
\in\mathbb{R}^R.
\end{equation}
Hence, by~\eqref{newnews}, we have
\begin{equation}\label{Vz}
    V_{M,r}\boldsymbol{z}_j
=
\frac{1}{\sqrt{\alpha\mu_j}}\,V_{M,r}V_{M,r}^\top \boldsymbol{u}_j
=
\sqrt{\alpha\mu_j}\,\boldsymbol{u}_j
\end{equation}
and
\begin{eqnarray*}
\left\langle \boldsymbol{z}_i,\boldsymbol{z}_j\right\rangle
&=&
\frac{1}{\sqrt{\alpha\mu_i}\sqrt{\alpha\mu_j}}
\left\langle V_{M,r}^\top \boldsymbol{u}_i,V_{M,r}^\top \boldsymbol{u}_j\right\rangle\\
&=&
\frac{1}{\sqrt{\alpha\mu_i}\sqrt{\alpha\mu_j}}
\left\langle \boldsymbol{u}_i,V_{M,r}V_{M,r}^\top \boldsymbol{u}_j\right\rangle
=
\frac{1}{\sqrt{\alpha\mu_i}\sqrt{\alpha\mu_j}}
\left\langle \boldsymbol{u}_i,\alpha\mu_j\boldsymbol{u}_j\right\rangle
=
\delta_{ij}.    
\end{eqnarray*}
Thus, the vectors $\boldsymbol{z}_1,\dots,\boldsymbol{z}_M$ are orthonormal in $\mathbb{R}^R$. Let now
\[
\left\{\boldsymbol{y}_1,\dots,\boldsymbol{y}_{R-M}\right\}
\]
be an orthonormal basis of $\ker\left(V_{M,r}\right)$. We show that each $\boldsymbol{z}_j$ is orthogonal to $\ker\left(V_{M,r}\right)$. Let
$\boldsymbol{y}\in\ker\left(V_{M,r}\right)$. Then, using~\eqref{defz}, we have
\begin{equation}\label{ortprop}
    \left\langle \boldsymbol{z}_j,\boldsymbol{y}\right\rangle
=
\frac{1}{\sqrt{\alpha\mu_j}}
\left\langle V_{M,r}^\top \boldsymbol{u}_j,\boldsymbol{y}\right\rangle
=
\frac{1}{\sqrt{\alpha\mu_j}}
\left\langle \boldsymbol{u}_j,V_{M,r}\boldsymbol{y}\right\rangle
=
0.
\end{equation}
Therefore
\[
\left\{\boldsymbol{y}_1,\dots,\boldsymbol{y}_{R-M},\boldsymbol{z}_1,\dots,\boldsymbol{z}_M\right\}
\]
is an orthonormal basis of $\mathbb{R}^R$.

For every $\ell=1,\dots,R-M$, since $\boldsymbol{y}_\ell\in\ker\left(V_{M,r}\right)$, we have
\[
\mathcal K
\begin{bmatrix}
\boldsymbol{y}_\ell\\
\boldsymbol{0}
\end{bmatrix}
=
\begin{bmatrix}
\alpha\boldsymbol{y}_\ell\\
V_{M,r}\boldsymbol{y}_\ell
\end{bmatrix}
=
\begin{bmatrix}
\alpha\boldsymbol{y}_\ell\\
\boldsymbol{0}
\end{bmatrix}
=
\alpha
\begin{bmatrix}
\boldsymbol{y}_\ell\\
\boldsymbol{0}
\end{bmatrix}.
\]
Therefore
\[
\begin{bmatrix}
\boldsymbol{y}_1\\
\boldsymbol{0}
\end{bmatrix},
\dots,
\begin{bmatrix}
\boldsymbol{y}_{R-M}\\
\boldsymbol{0}
\end{bmatrix}
\]
are linearly independent eigenvectors of $\mathcal K$ associated with the eigenvalue $\alpha$. Hence
$\alpha$ is an eigenvalue of $\mathcal K$ with multiplicity at least $R-M$.

For each $j=1,\dots,M$, consider the two-dimensional subspace
\[
\mathcal E_j
=
\operatorname{span}\left\{
\begin{bmatrix}
\boldsymbol{z}_j\\
\boldsymbol{0}
\end{bmatrix},
\begin{bmatrix}
\boldsymbol{0}\\
\boldsymbol{u}_j
\end{bmatrix}
\right\}
\subset\mathbb{R}^{R+M}.
\]
By~\eqref{Vz} and~\eqref{defz}, we obtain
\[
\mathcal K
\begin{bmatrix}
\boldsymbol{z}_j\\
\boldsymbol{0}
\end{bmatrix}
=
\begin{bmatrix}
\alpha\boldsymbol{z}_j\\
V_{M,r}\boldsymbol{z}_j
\end{bmatrix}
=
\begin{bmatrix}
\alpha\boldsymbol{z}_j\\
\sqrt{\alpha\mu_j}\,\boldsymbol{u}_j
\end{bmatrix},
\]
and
\[
\mathcal K
\begin{bmatrix}
\boldsymbol{0}\\
\boldsymbol{u}_j
\end{bmatrix}
=
\begin{bmatrix}
V_{M,r}^\top \boldsymbol{u}_j\\
\boldsymbol{0}
\end{bmatrix}
=
\begin{bmatrix}
\sqrt{\alpha\mu_j}\,\boldsymbol{z}_j\\
\boldsymbol{0}
\end{bmatrix}.
\]
Thus $\mathcal E_j$ is invariant under $\mathcal K$. Let
\[
T_j:\mathcal E_j\to\mathcal E_j
\]
be the restriction of $\mathcal K$ to $\mathcal E_j$, namely
\[
T_j\left(\boldsymbol{v}\right)=\mathcal K\boldsymbol{v},
\qquad \boldsymbol{v}\in\mathcal E_j.
\]
With respect to the orthonormal basis
\[
\left\{
\begin{bmatrix}
\boldsymbol{z}_j\\
\boldsymbol{0}
\end{bmatrix},
\begin{bmatrix}
\boldsymbol{0}\\
\boldsymbol{u}_j
\end{bmatrix}
\right\},
\]
the matrix representing $T_j$ is
\[
\begin{bmatrix}
\alpha & \sqrt{\alpha\mu_j}\\
\sqrt{\alpha\mu_j} & 0
\end{bmatrix}.
\]
Therefore, the eigenvalues of $T_j$ are the solution of the characteristic equation
\[
\lambda^2-\alpha\lambda-\alpha\mu_j=0,
\]
namely
\[
\lambda_j^\pm
=
\frac{\alpha\pm\sqrt{\alpha^2+4\alpha\mu_j}}{2}.
\]
Since $T_j$ is the restriction of $\mathcal K$ to the invariant subspace $\mathcal E_j$, every eigenvalue
of $T_j$ is an eigenvalue of $\mathcal K$. Hence
\[
\lambda_j^\pm\in\operatorname{spec}\left(\mathcal K\right),
\qquad j=1,\dots,M.
\]
Moreover, we have
\[
\lambda_j^+>0,
\qquad
\lambda_j^-<0.
\]
We now verify that these are all the eigenvalues of $\mathcal K$. Using~\eqref{ortprop}, the family
\[
\left\{
\begin{bmatrix}
\boldsymbol{y}_\ell\\
\boldsymbol{0}
\end{bmatrix}
:\ell=1,\dots,R-M
\right\}
\cup
\left\{
\begin{bmatrix}
\boldsymbol{z}_j\\
\boldsymbol{0}
\end{bmatrix},
\begin{bmatrix}
\boldsymbol{0}\\
\boldsymbol{u}_j
\end{bmatrix}
:j=1,\dots,M
\right\}
\]
is orthonormal in $\mathbb{R}^{R+M}$, and its cardinality is
\[
\left(R-M\right)+2M=R+M.
\]
Hence it is a basis of $\mathbb{R}^{R+M}$. It follows that
\[
\operatorname{spec}\left(\mathcal K\right)
=
\left\{\alpha\right\}^{R-M}
\cup
\left\{\lambda_j^+,\lambda_j^-:j=1,\dots,M\right\}.
\]

Since $\mathcal K$ is symmetric, its singular values coincide with the absolute values of its eigenvalues. Therefore, for $j=1,\dots,M$, we have
\[
\left|\lambda_j^+\right|
=
\frac{\alpha+\sqrt{\alpha^2+4\sigma_j\left(V_{M,r}\right)^2}}{2},
\qquad
\left|\lambda_j^-\right|
=
\frac{\sqrt{\alpha^2+4\sigma_j\left(V_{M,r}\right)^2}-\alpha}{2}.
\]
Moreover, $\alpha$ is a singular value of $\mathcal K$ with multiplicity $R-M$. Finally, the largest singular value of $\mathcal K$ is \[ \frac{\alpha+\sqrt{\alpha^2+4\sigma_{\max}\left(V_{M,r}\right)^2}}{2}, \] while the smallest one is \[ \min\left\{ \alpha,\, \frac{\sqrt{\alpha^2+4\sigma_{\min}\left(V_{M,r}\right)^2}-\alpha}{2} \right\}. \]
Taking the ratio gives~\eqref{eq:KKT_cond_design}.
\end{proof}

\section{A quasi-optimality estimate under a discrete norming assumption}
\label{secQuasiOp}

We now assume that the interpolation set $X_M$ is unisolvent for $\Pi_m\left(\mathbb{S}^2\right)$. Hence, for any $g\in C\left(\mathbb{S}^2\right)$, there exists a unique polynomial
$p_m[g]\in \Pi_m\left(\mathbb{S}^2\right)$ satisfying
\[
p_m[g]\left(\boldsymbol{x}_j'\right)=g\left(\boldsymbol{x}_j'\right),
\qquad j=1,\dots,M.
\]
Let $\left\{\ell_j\right\}_{j=1}^M\subset \Pi_m\left(\mathbb{S}^2\right)$ denote the
associated Lagrange basis, characterized by
\[
\ell_j\left(\boldsymbol{x}_k'\right)=\delta_{jk},
\qquad j,k=1,\dots,M.
\]
The interpolation operator is therefore given by
\[
p_m[g](\boldsymbol{x})=\sum_{j=1}^M g\left(\boldsymbol{x}_j'\right)\ell_j(\boldsymbol{x}).
\]
We denote its Lebesgue constant by
\[
\Lambda_M
:=
\sup_{\boldsymbol{x}\in\mathbb{S}^2}
\sum_{j=1}^M \left|\ell_j(\boldsymbol{x})\right|.
\]
The definition of $\Lambda_M$ implies that
\begin{equation}\label{eq:IM_stability}
\|p_m[g]\|_{L^\infty(\mathbb{S}^2)}
\leq
\Lambda_M \|g\|_{L^\infty(\mathbb{S}^2)}.
\end{equation}
Moreover, since $p_m[\cdot]$ is a projector onto
$\Pi_m\left(\mathbb{S}^2\right)$, we get
\begin{equation}\label{eq:IM_error1}
\left\|g-p_m[g]\right\|_{L^\infty\left(\mathbb{S}^2\right)}
\le
(1+\Lambda_M)\inf_{q\in \Pi_m\left(\mathbb{S}^2\right)}
\|g-q\|_{L^\infty\left(\mathbb{S}^2\right)}.
\end{equation}
We assume that the sampling set satisfies 
\begin{equation}\label{eq:lower_MZ}
\frac{1}{N}\sum_{\boldsymbol{x}_i\in X_N} \left|p\left(\boldsymbol{x}_i\right)\right|^2
\geq
\alpha_r \left\|p\right\|_{L^2\left(\mathbb{S}^2\right)}^2,
\qquad
\forall p\in \Pi_r\left(\mathbb{S}^2\right),
\end{equation}
for a suitable constant $\alpha_r>0$. In the following we will use the classical bound
\begin{equation}\label{boundsa}
    \|g\|_{L^2\left(\mathbb{S}^2\right)}^2
=
\int_{\mathbb{S}^2} |g(\boldsymbol{x})|^2d\omega(\boldsymbol{x})
\leq
\|g\|_{L^\infty\left(\mathbb{S}^2\right)}^2
\int_{\mathbb{S}^2} d\omega(\boldsymbol{x})
=
4\pi\|g\|_{L^\infty\left(\mathbb{S}^2\right)}^2, \quad g\in C\left(\mathbb{S}^2\right).
\end{equation}

\begin{remark}
If $X_N$ is a spherical $k$-design with $k\ge 2r$, then
Theorem~\ref{thm:design_orthogonality} gives
\[
\frac{1}{N}\sum_{\boldsymbol{x}_i\in X_N}
\left|p\left(\boldsymbol{x}_i\right)\right|^2
=
\frac{1}{4\pi}\|p\|_{L^2\left(\mathbb{S}^2\right)}^2,
\qquad
\forall p\in \Pi_r\left(\mathbb{S}^2\right).
\]
Thus~\eqref{eq:lower_MZ} holds with $\alpha_r=\frac{1}{4\pi}$.
\end{remark}

In the next result, we prove a quasi-optimality estimate for the interpolation--regression approximant on the sphere.

\begin{theorem}\label{thm:projection_structure_sharp_quasi}
Assume that $X_M$ is unisolvent for $\Pi_m\left(\mathbb{S}^2\right)$ and that the sampling set $X_N$ satisfies~\eqref{eq:lower_MZ}. Then, for any $f\in C\left(\mathbb{S}^2\right)$, the following statements hold.
\begin{itemize}
\item[(a)] The polynomial $\hat p_r[f]-p_m[f]$ satisfies
\begin{equation}\label{eq:projection_identity}
\left\langle \left(f-p_m[f]\right)-\left(\hat p_r[f]-p_m[f]\right), q\right\rangle_{X_N}=0,
\qquad \forall q\in \mathcal V_{r,M}.
\end{equation}

\item[(b)] Using the norm~\eqref{normh}, the following Pythagorean identity holds
\begin{equation}\label{eq:pythagorean_identity}
\left\|f-\hat p_r[f]\right\|_{2,X_N}^2
+
\left\|\hat p_r[f]-p_m[f]\right\|_{2,X_N}^2
=
\left\|f-p_m[f]\right\|_{2,X_N}^2,
\end{equation}

\item[(c)] The interpolation--regression approximant satisfies
\begin{equation}\label{eq:sharp_quasi_optimal_L2}
\left\|f-\hat p_r[f]\right\|_{L^2\left(\mathbb{S}^2\right)}
\le
\left(
\sqrt{4\pi}+\frac{1}{\sqrt{\alpha_r}}
\right)
\left(1+\Lambda_M\right)
\inf_{q\in \Pi_m\left(\mathbb{S}^2\right)}
\left\|f-q\right\|_{L^\infty\left(\mathbb{S}^2\right)}.
\end{equation}
\end{itemize}
\end{theorem}

\begin{proof}
We consider $f\in C\left(\mathbb{S}^2\right)$. Since $\Pi_m\left(\mathbb{S}^2\right)\subset \Pi_r\left(\mathbb{S}^2\right)$, we have $p_m[f]\in \mathcal A_{r,M}(f)$. By construction, also $\hat p_r[f]\in \mathcal A_{r,M}(f)$, and therefore
\begin{equation}\label{hatpcorr}
    \hat p_r[f]-p_m[f]\in \mathcal V_{r,M}.
\end{equation}
Let $q\in \mathcal V_{r,M}$. Using the linearity of the inner product, we have
\[
\left\langle f-\hat p_r[f], q\right\rangle_{X_N}=\left\langle \left(f-p_m[f]\right)-\left(\hat p_r[f]-p_m[f]\right), q\right\rangle_{X_N}. 
\]
Then, by Theorem~\ref{thmorth}, we get
\[
\left\langle \left(f-p_m[f]\right)-\left(\hat p_r[f]-p_m[f]\right), q\right\rangle_{X_N}
=0,
\]
which proves~\eqref{eq:projection_identity}.

We now prove~\eqref{eq:pythagorean_identity}. By~\eqref{eq:projection_identity} and~\eqref{hatpcorr}, we have
\begin{equation}\label{pytdes}
    \left\langle 
f-\hat p_r[f],
\hat p_r[f]-p_m[f]
\right\rangle_{X_N}=\left\langle 
\left(f-p_m[f]\right)-\left(\hat p_r[f]-p_m[f]\right),
 \hat p_r[f]-p_m[f]
\right\rangle_{X_N}
=0.
\end{equation}
Therefore
\begin{eqnarray*}
\left\|f-p_m[f]\right\|_{2,X_N}^2
&=&
\left\|
\left(f-\hat p_r[f]\right)+\left(\hat p_r[f]-p_m[f]\right)
\right\|_{2,X_N}^2 \\
&=&
\left\|f-\hat p_r[f]\right\|_{2,X_N}^2
+
\left\|\hat p_r[f]-p_m[f]\right\|_{2,X_N}^2 
+2\left\langle f-\hat p_r[f], \hat p_r[f]-p_m[f]\right\rangle_{X_N}\\
&=& \left\|f-\hat p_r[f]\right\|_{2,X_N}^2
+
\left\|\hat p_r[f]-p_m[f]\right\|_{2,X_N}^2.
\end{eqnarray*}
This proves~\eqref{eq:pythagorean_identity}.
Consequently, we can write
\begin{equation}\label{eq:projection_bound_disc}
\left\|f-p_m[f]\right\|_{2,X_N}\ge 
\left\|\hat p_r[f]-p_m[f]\right\|_{2,X_N}.
\end{equation}
Using~\eqref{normh} and applying the inequality~\eqref{eq:lower_MZ}, we get
\[
\left\|\hat p_r[f]-p_m[f]\right\|_{L^2\left(\mathbb{S}^2\right)}
\le
\frac{1}{\sqrt{N\alpha_r}}
\left\|\hat p_r[f]-p_m[f]\right\|_{2,X_N}.
\]
Combining this with~\eqref{eq:projection_bound_disc}, we obtain
\[
\left\|\hat p_r[f]-p_m[f]\right\|_{L^2\left(\mathbb{S}^2\right)}
\le
\frac{1}{\sqrt{N\alpha_r}}
\left\|f-p_m[f]\right\|_{2,X_N}.
\]
By definition of the $L^\infty$ norm, we have
\[
\left\|f-p_m[f]\right\|_{2,X_N}
\le
\sqrt{N}
\left\|f-p_m[f]\right\|_{L^\infty\left(\mathbb{S}^2\right)}.
\]
Then
\begin{equation*}
\left\|\hat p_r[f]-p_m[f]\right\|_{L^2\left(\mathbb{S}^2\right)}
\le
\frac{1}{\sqrt{\alpha_r}}
\left\|f-p_m[f]\right\|_{L^\infty\left(\mathbb{S}^2\right)}.
\end{equation*}
Hence, using the triangle inequality and~\eqref{boundsa},  we obtain
\begin{eqnarray*}
\left\|f-\hat p_r[f]\right\|_{L^2\left(\mathbb{S}^2\right)}
&\le&
\left\|f-p_m[f]\right\|_{L^2\left(\mathbb{S}^2\right)}
+
\left\|p_m[f]-\hat p_r[f]\right\|_{L^2\left(\mathbb{S}^2\right)} \\
&\le&
\sqrt{4\pi}
\left\|f-p_m[f]\right\|_{L^\infty\left(\mathbb{S}^2\right)}
+
\frac{1}{\sqrt{\alpha_r}}
\left\|f-p_m[f]\right\|_{L^\infty\left(\mathbb{S}^2\right)}.
\end{eqnarray*}
Using~\eqref{eq:IM_error1}, we conclude~\eqref{eq:sharp_quasi_optimal_L2}.
\end{proof}

\begin{remark}
The inequality~\eqref{eq:lower_MZ} imposes a nontrivial requirement on the
distribution of the sampling points and is not automatically satisfied by
arbitrary node configurations. Indeed, let $r\ge 1$ and consider, for any
$N\ge 1$, the equatorial set
\[
X_N=
\left\{
\boldsymbol{x}_j=
\left(
\cos\frac{2\pi j}{N},
\sin\frac{2\pi j}{N},
0
\right)
:\ j=0,\dots,N-1
\right\}
\subset \mathbb{S}^2 .
\]
For the polynomial $p(x,y,z)=z$, we have
$p\in\Pi_1\left(\mathbb{S}^2\right)\subset\Pi_r\left(\mathbb{S}^2\right)$ and
\[
p\left(\boldsymbol{x}_j\right)=0,
\qquad j=0,\dots,N-1.
\]
Thus
\[
\frac{1}{N}
\sum_{\boldsymbol{x}_i\in X_N}
\left|p\left(\boldsymbol{x}_i\right)\right|^2=0,
\]
whereas
\[
\|p\|_{L^2(\mathbb{S}^2)}^2
=
\int_{\mathbb{S}^2} z^2\,d\omega
=
\frac{4\pi}{3}
>0.
\]
\end{remark}

\section{A constructive antipodal strategy for the interpolation--regression scheme}
\label{sec:algorithm}
The antipodally symmetric framework developed in Section~\ref{sec:interp_sphere} requires an interpolation set that is itself compatible with the symmetry of the problem and, at the same time, satisfies the rank condition needed for the constrained least squares formulation. We describe here a constructive way to obtain such a set.

Throughout this section, let
\[
X_N=
\left\{
\boldsymbol{x}_1,\dots,\boldsymbol{x}_N
\right\}
\subset \mathbb{S}^2
\]
be a unisolvent set of nodes for $\Pi_n\left(\mathbb{S}^2\right)$.
We define the set
\[
\widetilde{X}_N
=
X_N\cup \left(-X_N\right),
\]
which is antipodally symmetric by construction. 
For simplicity, we assume that
\[
X_N\cap(-X_N)=\emptyset.
\]
For a fixed degree $r<n$, let $\widetilde{V}_{N,r}$ 
denote the Vandermonde matrix relative to $\widetilde{X}_N$. In the next result, we prove that the rank condition~\eqref{eq:rank_condition_section} is satisfied.

\begin{proposition}
Let $X_N\subset \mathbb{S}^2$ be unisolvent for $\Pi_n\left(\mathbb{S}^2\right)$, and let
\[
\widetilde{X}_N=X_N\cup \left(-X_N\right).
\]
If $r\leq n$, then
\[
\operatorname{rank}\left(\widetilde{V}_{N,r}\right)=R.
\]
\end{proposition}

\begin{proof}
Let $V_{N,r}$ denote the Vandermonde matrix obtained by evaluating a basis of $\Pi_r\left(\mathbb{S}^2\right)$ at the nodes of $X_N$. Since
\[
\Pi_r\left(\mathbb{S}^2\right)\subset \Pi_n\left(\mathbb{S}^2\right),
\]
the unisolvence of $X_N$ for $\Pi_n\left(\mathbb{S}^2\right)$ implies that the only polynomial in $\Pi_r\left(\mathbb{S}^2\right)$ vanishing at all nodes of $X_N$ is the zero polynomial. Hence 
\[
\operatorname{rank}\left(V_{N,r}\right)=R.
\]
The matrix $\widetilde{V}_{N,r}$ is obtained from $V_{N,r}$ by adding the rows corresponding to the antipodal nodes. Since adding rows cannot decrease the column rank, we have
\[
\operatorname{rank}\left(\widetilde{V}_{N,r}\right)= R.
\]
\end{proof}

Once the rank condition has been established for the completed sampling set, we turn to the selection of the interpolation nodes. Our goal is to extract from $\widetilde X_N$ an antipodally symmetric subset
\[
X_M\subset \widetilde{X}_N
\]
such that the rank condition~\eqref{rankM} is satisfied. The antipodal structure allows this condition to be expressed in terms of the even and odd components of the basis. Assuming to have ordered the real spherical harmonics according to the parity of the degree, we write
\[
\widetilde{V}_{N,r}
=
\left[
\widetilde{V}_{N,r}^{+}\;\;\widetilde{V}_{N,r}^{-}
\right].
\]
For each $j=1,\dots,N$, let $\boldsymbol{v}_{j}^{+}$ and $\boldsymbol{v}_{j}^{-}$ denote the rows of $\widetilde{V}_{N,r}^{+}$ and $\widetilde{V}_{N,r}^{-}$ corresponding to the node $\boldsymbol{x}_j$, respectively. By the parity relation~\eqref{parcond}, the corresponding rows corresponding to the antipodal node $-\boldsymbol{x}_j$ are $\boldsymbol{v}_{j}^{+}$ and $-\boldsymbol{v}_{j}^{-}$. Therefore, if
\[
J\subset\{1,\dots,N\},
\]
the Vandermonde matrix associated with the antipodal subset
\[
X(J)
=
\left\{
\boldsymbol{x}_j,-\boldsymbol{x}_j:\ j\in J
\right\}
\]
has the block form
\[
\widetilde{V}_{J,r}
=
\begin{bmatrix}
\widetilde V_{J,r}^{+} & \widetilde V_{J,r}^{-}\\
\widetilde V_{J,r}^{+} & -\widetilde V_{J,r}^{-}
\end{bmatrix},
\]
where $\widetilde V_{J,r}^{+}$ and $\widetilde V_{J,r}^{-}$ are the matrices obtained by collecting the rows $\boldsymbol{v}_{j}^{+}$ and $\boldsymbol{v}_{j}^{-}$ for $j\in J$, respectively.

The following result is crucial for the construction below.

\begin{proposition}\label{prop:pair_rank_characterization}
Let $J\subset\{1,\dots,N\}$ and set $K=|J|$. Let $R_+$ and $R_-$ denote the number of columns of
$\widetilde V_{J,r}^{+}$ and $\widetilde V_{J,r}^{-}$, respectively. Assume that
\[
K\le \min\left\{R_+,R_-\right\}.
\]
Then
\[
\operatorname{rank}\left(\widetilde V_{J,r}\right)=2K
\]
if and only if
\[
\operatorname{rank}\left(\widetilde V_{J,r}^{+}\right)=K
\qquad\text{and}\qquad
\operatorname{rank}\left(\widetilde V_{J,r}^{-}\right)=K.
\]
\end{proposition}
\begin{proof}
By construction
\[
\widetilde V_{J,r}
=
\begin{bmatrix}
\widetilde V_{J,r}^{+} & \widetilde V_{J,r}^{-}\\
\widetilde V_{J,r}^{+} & -\widetilde V_{J,r}^{-}
\end{bmatrix}\in\mathbb{R}^{2K\times R},
\]
where
\[
\widetilde V_{J,r}^{+}\in\mathbb{R}^{K\times R_+},
\qquad
\widetilde V_{J,r}^{-}\in\mathbb{R}^{K\times R_-},
\qquad
R_+ + R_- = R.
\]
We define
\[
T:=\frac{1}{2}
\begin{bmatrix}
I_K & I_K\\
I_K & -I_K
\end{bmatrix}
\in\mathbb{R}^{2K\times 2K},
\]
where $I_K\in \mathbb{R}^{K\times K}$ is the identity matrix.
Then
\[
T\widetilde V_{J,r}
=
\begin{bmatrix}
\widetilde V_{J,r}^{+} & 0\\
0 & \widetilde V_{J,r}^{-}
\end{bmatrix}.
\]
Since left multiplication by $T$ preserves rank, we obtain
\[
\operatorname{rank}\left(\widetilde V_{J,r}\right)
=
\operatorname{rank}\left(\widetilde V_{J,r}^{+}\right)
+
\operatorname{rank}\left(\widetilde V_{J,r}^{-}\right).
\]
The conclusion follows immediately, since both reduced matrices have $K$ rows.
\end{proof}

In Proposition~\ref{prop:pair_rank_characterization}, we show that it is sufficient to work with the reduced even and odd matrices when constructing admissible antipodal interpolation sets. This leads to a two-step construction.

\medskip

\noindent
\textbf{Step 1.}
We start with $J=\emptyset$ and add antipodal pairs iteratively. An index $j\notin J$ is admissible if
\[
\operatorname{rank}\left(\widetilde V_{J\cup\{j\},r}^{+}\right)
=
\operatorname{rank}\left(\widetilde V_{J,r}^{+}\right)+1
\]
and
\[
\operatorname{rank}\left(\widetilde V_{J\cup\{j\},r}^{-}\right)
=
\operatorname{rank}\left(\widetilde V_{J,r}^{-}\right)+1.
\]
Among all admissible candidates, one selects the pair according to a stability criterion. In the implementation used in the numerical experiments, we employ a greedy criterion based on the smallest singular values of the updated reduced matrices. The iteration is stopped when no further admissible pair can be added. The resulting set
\[
J_{\max}=\left\{j_1,\dots,j_{K_{\max}}\right\}
\]
defines a \emph{maximal} admissible antipodal system
\[
X_{\max}
=
\left\{
\boldsymbol{x}_{j_1},\dots,\boldsymbol{x}_{j_{K_{\max}}},
-\boldsymbol{x}_{j_1},\dots,-\boldsymbol{x}_{j_{K_{\max}}}
\right\},
\qquad M_{\max}=2K_{\max}.
\]

\medskip

\noindent
\textbf{Step 2.}
The maximal cardinality $M_{\max}$ need not coincide with a value of the form $(m+1)^2$, where $(m+1)^2=\dim \Pi_m\left(\mathbb{S}^2\right)$. Consider the set
\[
\mathcal{S}:=
\left\{
s\in\mathbb{N}:\ (s+1)^2\leq M_{\max},\ s<r,\ (s+1)^2 \ \text{even}
\right\}.
\]
If $\mathcal{S}=\emptyset$, we set $\widetilde X_M=\emptyset$, and the procedure reduces to the unconstrained least squares approximation over $\widetilde X_N$. Otherwise, we define
\[
m=\max \mathcal{S},
\qquad
M=(m+1)^2.
\]
The parity constraint is necessary because the interpolation set must remain antipodally symmetric. Since $M\leq M_{\max}$, we obtain the final interpolation set by removing some antipodal pairs from $X_{\max}$. This reduction is again performed greedily. At each step, we remove one antipodal pair only when the two reduced matrices remain of full row rank. Among the admissible choices, we select the removal that maximizes
\[
\min\left\{
\sigma_{\min}\left(\widetilde V_{J,r}^{+}\right),
\sigma_{\min}\left(\widetilde V_{J,r}^{-}\right)
\right\}.
\]
The final output is an antipodally symmetric subset $\widetilde{X}_M\subset X_{\max}\subset \widetilde{X}_N$
with cardinality $M=(m+1)^2$ and $\operatorname{rank}\left(\widetilde{V}_{M,r}\right)=M.$

The construction therefore proceeds in two stages. It first builds a maximal admissible antipodal set for the prescribed degree $r$. It then selects from it a subset whose cardinality has the form $(m+1)^2$, with $m<r$, while preserving the required rank condition.

\begin{remark}
The procedure should be understood as a constructive strategy. Specifically, its purpose is to produce an antipodally symmetric interpolation set with the required rank property, rather than to solve a global optimization problem.
\end{remark}

\begin{remark}
In the one-dimensional setting considered in~\cite{OTC:JCAM:2015}, the interpolation and reconstruction degrees can be related through explicit asymptotic rules. On the sphere, the situation is more dependent on the geometry of the nodes, and no analogous rule is available. 
\end{remark}

The complete procedure is summarized in Algorithm~\ref{alg:antipodal_strategy}.

\begin{algorithm}[H]
\caption{Constructive antipodal strategy}
\label{alg:antipodal_strategy}
\begin{algorithmic}[1]
\Require A unisolvent set $X_N\subset\mathbb{S}^2$ for $\Pi_n\left(\mathbb{S}^2\right)$; reconstruction degree $r<n$
\Ensure An antipodally symmetric set $\widetilde{X}_M\subset \widetilde{X}_N$ such that either $\widetilde{X}_M=\emptyset$, or $\left|\widetilde{X}_M\right|=M=(m+1)^2$ and $\operatorname{rank}\left(\widetilde{V}_{M,r}\right)=M$

\State Construct the antipodal completion
\[
\widetilde{X}_N = X_N\cup\left(-X_N\right)
\]

\State Build the Vandermonde matrix $\widetilde{V}_{N,r}$ and split it into even and odd blocks
\[
\widetilde{V}_{N,r}=\left[\widetilde{V}_{N,r}^{+}\;\;\widetilde{V}_{N,r}^{-}\right]
\]

\State Build the reduced pair matrices $\widetilde V_{N,r}^{+}$ and $\widetilde V_{N,r}^{-}$ associated with the antipodal pairs $\left\{\boldsymbol{x}_j,-\boldsymbol{x}_j\right\}$, $j=1,\dots,N$

\State Initialize $J\gets \emptyset$

\While{there exists $j\notin J$ such that
\[
\operatorname{rank}\left(\widetilde V_{J\cup\{j\},r}^{+}\right)
=
\operatorname{rank}\left(\widetilde V_{J,r}^{+}\right)+1
\]
and
\[
\operatorname{rank}\left(\widetilde V_{J\cup\{j\},r}^{-}\right)
=
\operatorname{rank}\left(\widetilde V_{J,r}^{-}\right)+1
\]}
    \State Select an admissible index $j$ according to a stability criterion
    \State $J\gets J\cup\{j\}$
\EndWhile

\State Set $J_{\max}\gets J$, \quad $M_{\max}\gets 2|J_{\max}|$

\State Define
\[
\mathcal{S}
=
\left\{
s\in\mathbb{N}:\ (s+1)^2\leq M_{\max},\ s<r,\ (s+1)^2 \ \text{even}
\right\}
\]

\If{$\mathcal{S}=\emptyset$}
    \State Set $\widetilde{X}_M\gets \emptyset$
    \State \Return $\widetilde{X}_M$
\EndIf

\State Set
\[
m = \max \mathcal{S}
\]

\State Set $M\gets (m+1)^2$ and $K\gets M/2$

\State Starting from $J_{\max}$, remove antipodal pairs one at a time until exactly $K$ pairs remain, preserving full row rank of both reduced matrices and favoring numerically stable choices

\State Denote by $J_M$ the resulting set of indices

\State Construct
\[
\widetilde{X}_M
=
\left\{
\boldsymbol{x}_j,-\boldsymbol{x}_j:\ j\in J_M
\right\}
\]

\State \Return $\widetilde{X}_M$
\end{algorithmic}
\end{algorithm}

\subsection{Practical selection of the reconstruction degree}
\label{subsec:parameter_selection}

The approximant introduced in~\eqref{eq:cls_problem} depends on the degree $r$. Once $r$ is fixed, the admissible antipodal set is constructed and the interpolation degree $m$ is determined. Thus $m$ is not prescribed independently. Since the admissible cardinality depends on the geometry of the nodes, we do not impose a closed-form rule for choosing $r$. In practice, $r$ may be selected by a validation procedure. Given a set of test functions
\[\mathcal{F}_{\mathrm{val}}
=
\left\{
f_1,\dots,f_S
\right\}
\]
a family of node sets $\left\{
X_{N_j}^{(j)}
\right\}_{j=1}^J$ and a finite set of degrees
\[
\mathcal{R}
=
\left\{
r_1,\dots,r_L
\right\}
\]
we compute, for each $r\in\mathcal R$, the corresponding interpolation--regression approximants and define
 \[
E(r)
:=
\sum_{s=1}^S \sum_{j=1}^J
\left\|
f_s-\hat p_r^{(j)}\left[f_s\right]
\right\|_{\infty}.
\]
We then choose
\[
r^\star \in \operatorname*{argmin}_{r\in\mathcal R} E(r).
\]

\begin{algorithm}[H]
\caption{Validation-based selection of the reconstruction degree}
\label{alg:degree_tuning}
\begin{algorithmic}[1]
\Require Validation functions $\mathcal{F}_{\mathrm{val}}=\left\{f_1,\dots,f_S\right\}$; node families $\left\{X_{N_j}^{(j)}\right\}_{j=1}^J$; candidate degrees $\mathcal{R}$
\Ensure Selected reconstruction degree $r^\star$

\State $E_{\min}\gets +\infty$, \quad $r^\star\gets \mathrm{undefined}$

\ForAll{$r\in\mathcal{R}$}
    \State $E(r)\gets 0$
    \For{$s=1$ to $S$}
        \For{$j=1$ to $J$}
            \State Construct the antipodal completion of $X_{N_j}^{(j)}$
            \State Apply Algorithm~\ref{alg:antipodal_strategy} with degree $r$
            \State Compute the corresponding interpolation--regression approximant $\hat p_r^{(j)}\left[f_s\right]$
            \State $E(r)\gets E(r)+\left\|f_s-\hat p_r^{(j)}\left[f_s\right]\right\|_{\infty}$
        \EndFor
    \EndFor
    \If{$E(r)<E_{\min}$}
        \State $E_{\min}\gets E(r)$
        \State $r^\star\gets r$
    \EndIf
\EndFor

\State \Return $r^\star$
\end{algorithmic}
\end{algorithm}

\section{Numerical experiments}\label{sec:numtest}
In this section, we report two families of numerical experiments aimed at testing the accuracy of the interpolation--regression operator~\eqref{eq:cls_problem}. We first study the method in the antipodally symmetric setting, where the construction of Section~\ref{sec:algorithm} can be applied directly. We then consider a general, non-symmetric configuration, in order to test the method under less structured node distributions.

We consider the following test functions on the unit sphere
\[
f_1(x,y,z)= e^{xz}, \qquad 
f_2(x,y,z)= \sin(3x)+\frac12\cos(2yz),
\]
\[
f_3(x,y,z)= \frac{1}{3+x^2+2y^2+4z^2}, \qquad
f_4(x,y,z)= \frac{1}{(x+1.5)^2+(y+1.5)^2+(z+1.5)^2}.
\]
For each test function $f_i$, we compute the uniform error
\[
E_{i,\infty}(r;X_N,X_M)
:=
\max_{\mathbf y\in Y_L^{\mathrm{val}}}
\left|f_i(\mathbf y)-\hat{p}_r\left[f_i\right](\mathbf y)\right|,
\qquad i=1,2,3,4,
\]
where 
\[
Y_L^{\mathrm{val}}\subset \mathbb S^2
\]
is a validation set generated by a spherical Fibonacci rule; see~\cite{FNI:JOP:2004}.

In the following, $n\in\mathbb{N}$ is used to fix the number of sampling nodes, and we set
\[
N=(n+1)^2.
\]
We perform the experiments for
\[
n\in\{20,30,40,50,60\}.
\]
For each value of $n$, the reconstruction degree $r$ is chosen such that $r<n$.

\subsection{Antipodal setting}
Let $X_N^{\mathrm{Xu}}$ be the unisolvent node set introduced in~\cite{PIO:ACOM:2004}. 
For the antipodal tests, we consider 
\[
\widetilde X_N=X_N^{\mathrm{Xu}}\cup\left(-X_N^{\mathrm{Xu}}\right)
\]
as sampling set. For each value of $n$, the reconstruction degree $r<n$ is chosen by the validation procedure described in Section~\ref{sec:algorithm}. The interpolation subset $X_M$ is then extracted from $\widetilde X_N$ using the constructive antipodal strategy of Section~\ref{sec:algorithm}. In Fig.~\ref{fig2} we show the resulting configurations for $n=40,50,60$. 
\begin{figure}
  \centering
\includegraphics[width=0.32\textwidth]{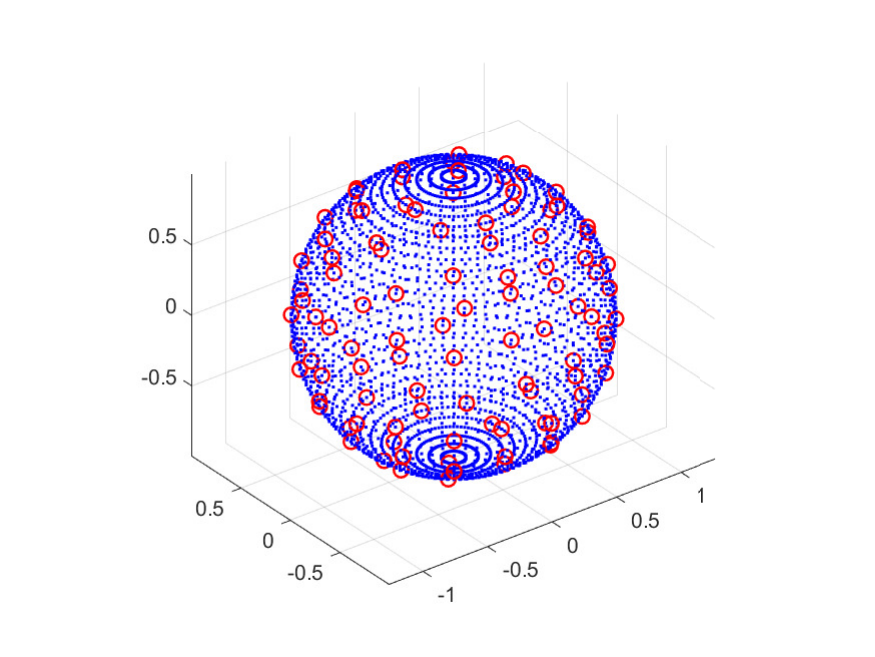} 
\includegraphics[width=0.32\textwidth]{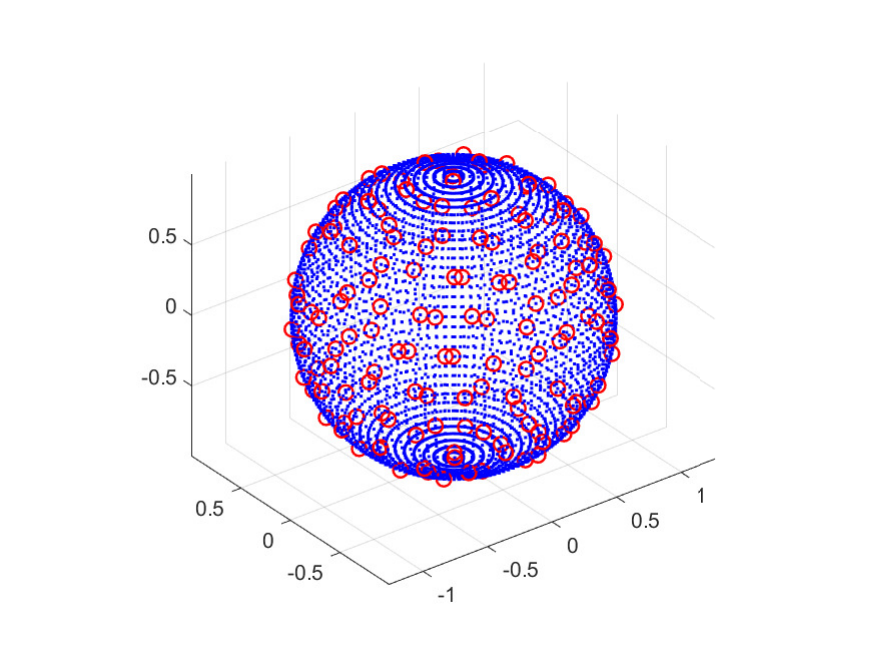} 
\includegraphics[width=0.32\textwidth]{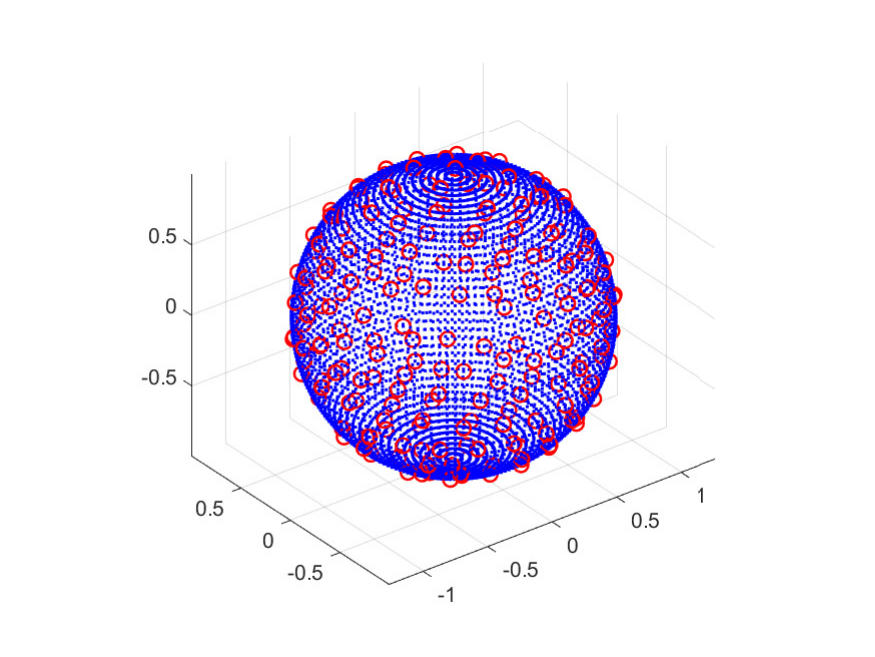} 
\caption{Antipodal distribution of the nodes $\widetilde{X}_N$ (blue) and of the selected subset $X_M$ (red) for $n=40$ (left), $n=50$ (middle), and $n=60$ (right).}
 \label{fig2}
\end{figure}

In Figure~\ref{fig1}, we report the trend of the maximum approximation error 
\[
E_{i,\infty}\left(r;\widetilde{X}_N,X_M\right)
=
\max_{\mathbf y\in Y_L^{\mathrm{val}}}
\left|f_i(\mathbf y)-\hat{p}_r\left[f_i\right](\mathbf y)\right|,
\qquad i=1,2,3,4.
\]
\begin{figure}
  \centering
\includegraphics[width=0.49\textwidth]{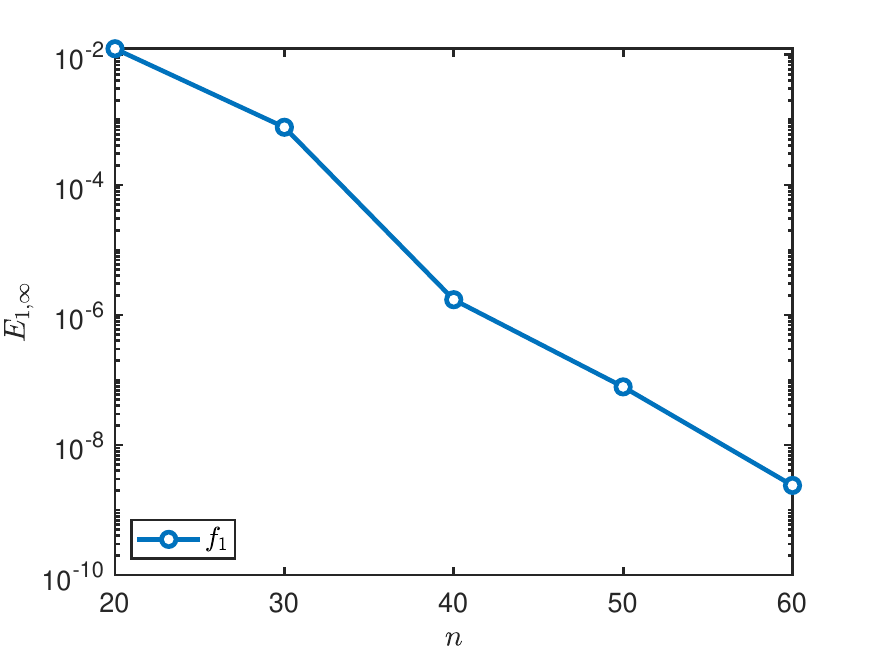} 
\includegraphics[width=0.49\textwidth]{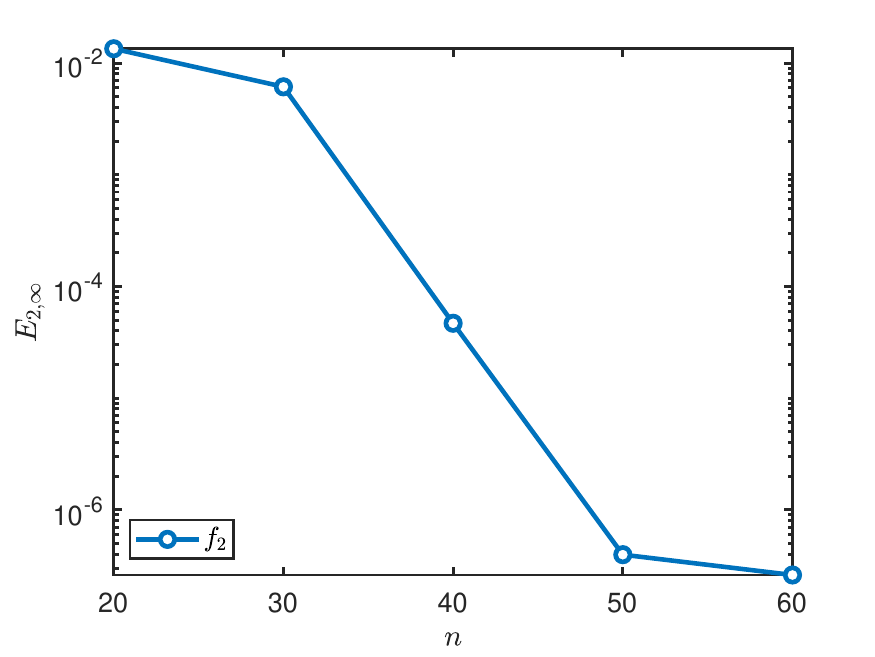} 
\includegraphics[width=0.49\textwidth]{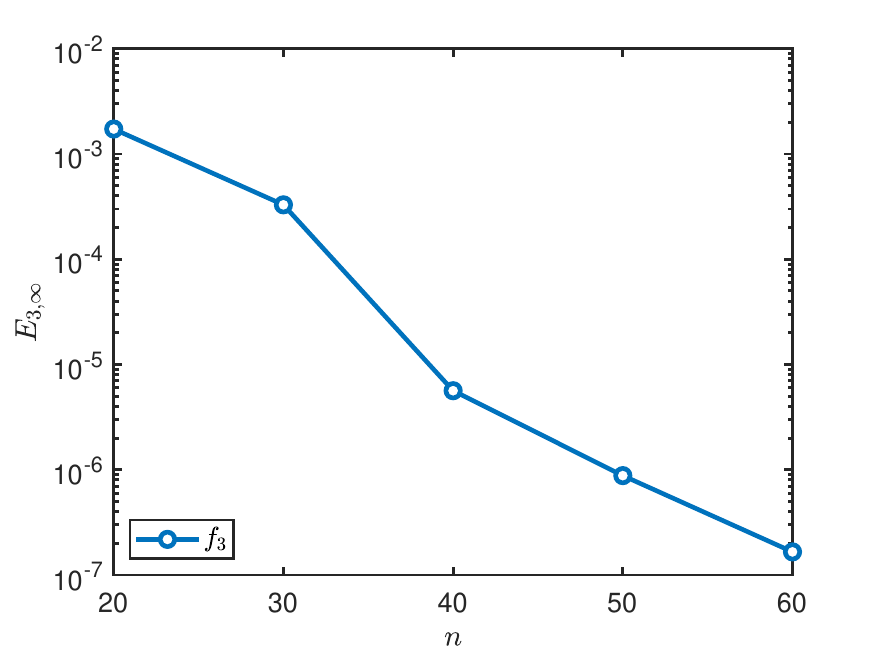} 
\includegraphics[width=0.49\textwidth]{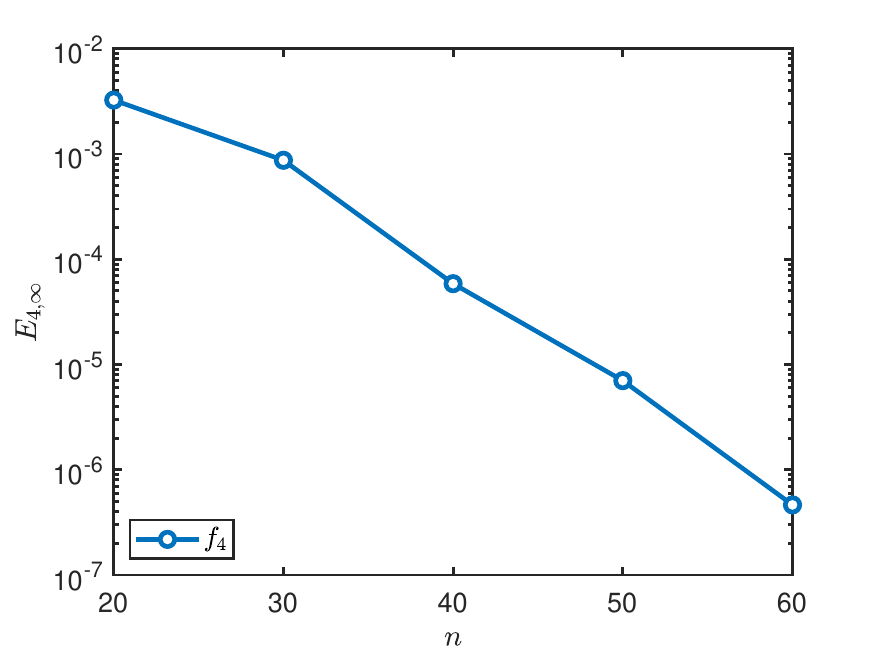} 
\caption{Trend of the maximum error $E_{i,\infty}\left(r;\widetilde{X}_N,X_M\right)$ as $n$ varies, for the test functions $f_1$ (top left), $f_2$ (top right), $f_3$ (bottom left), and $f_4$ (bottom right).}
 \label{fig1}
\end{figure}
From these plots, we can observe that the error decreases as $n$ increases for all four test functions. The decay is particularly regular for $f_1$, while $f_2$ shows a fast initial improvement followed by a slower decrease. The rational test functions $f_3$ and $f_4$ display a more gradual, but still stable, convergence trend.

\subsection{Non-antipodal setting}
We now consider the general setting, where no symmetry is imposed. For each $n$, we set
\[
X_N = X_N^{\mathrm{Xu}},
\]
and we select an interpolation subset $X_M \subset X_N$ of cardinality
\[
M=\dim \Pi_m\left(\mathbb S^2\right)=(m+1)^2,
\]
by means of a discrete approximate Fekete procedure. In all experiments, the degrees $m$ and $r$ are chosen as
\[
m = \left\lfloor \frac{n}{4}\right\rfloor + 1,
\qquad
r = m + \left\lfloor \sqrt{2m}\right\rfloor,
\]
so that 
\[
m < r < n.
\]
In Fig.~\ref{fig3}, we show the sampling sets and the corresponding interpolation subsets for $n=40,50,60$.
\begin{figure}
  \centering
  \includegraphics[width=0.32\textwidth]{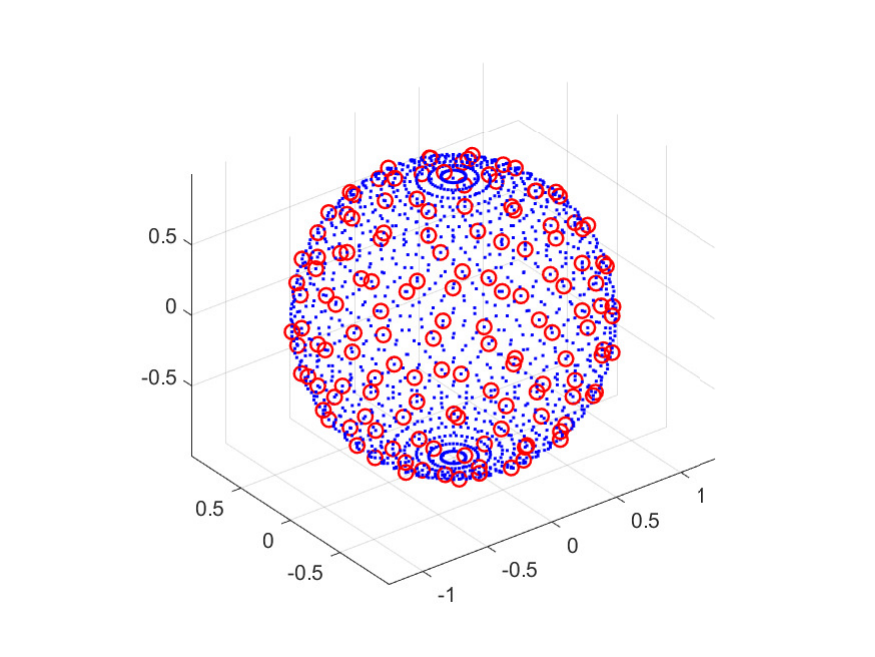}
  \includegraphics[width=0.32\textwidth]{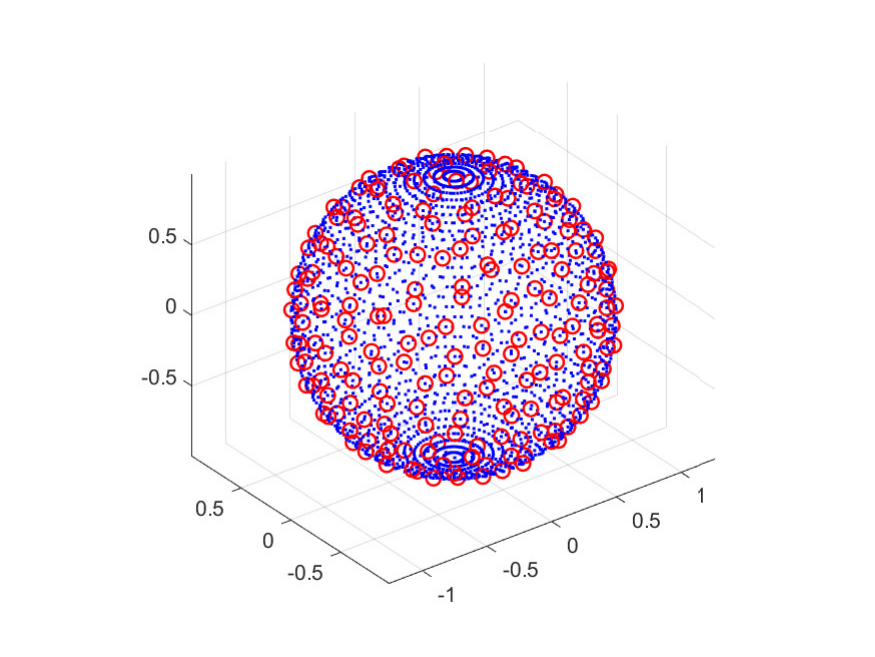}
  \includegraphics[width=0.32\textwidth]{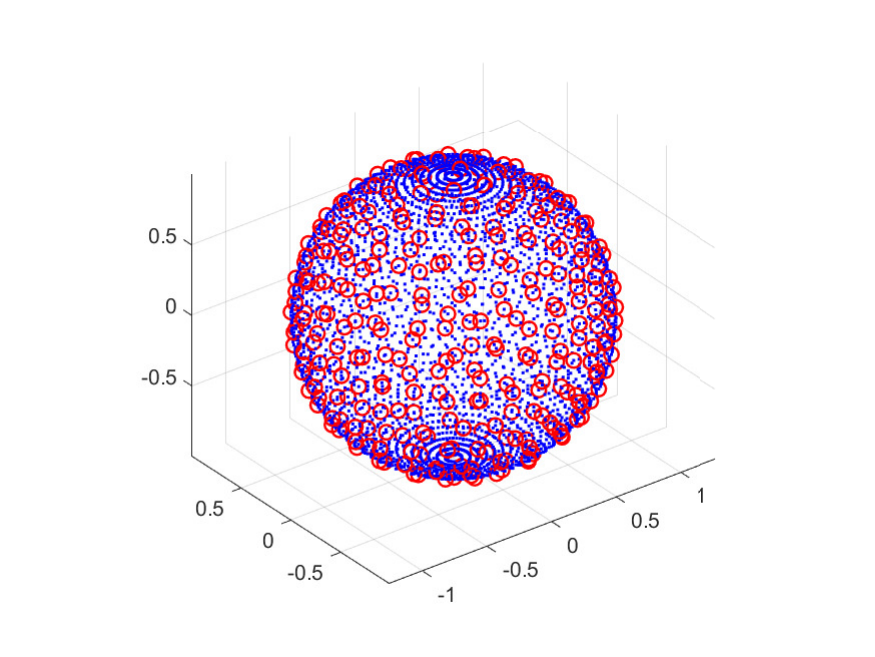}
  \caption{Sampling sets $X_N$ (blue) and selected approximate Fekete interpolation subsets $X_M$ (red) for $n=40$ (left), $n=50$ (centre), and $n=60$ (right). }
  \label{fig3}
\end{figure}

In Fig.~\ref{fig4}, we report the trend of the maximum approximation error $E_{i,\infty}\left(r;X_N,X_M\right)$.
\begin{figure}
  \centering
  \includegraphics[width=0.49\textwidth]{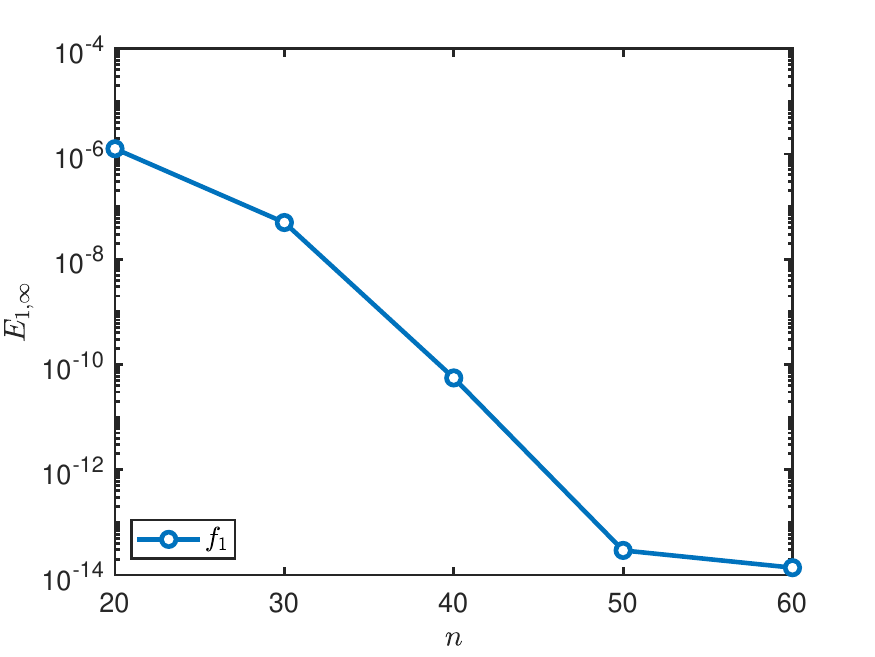}
  \includegraphics[width=0.49\textwidth]{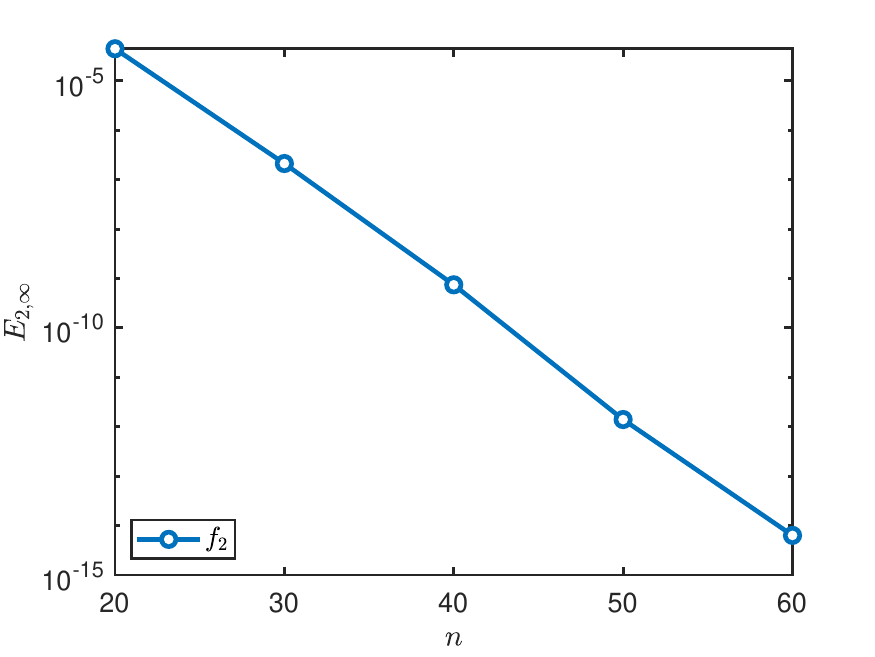}
  \includegraphics[width=0.49\textwidth]{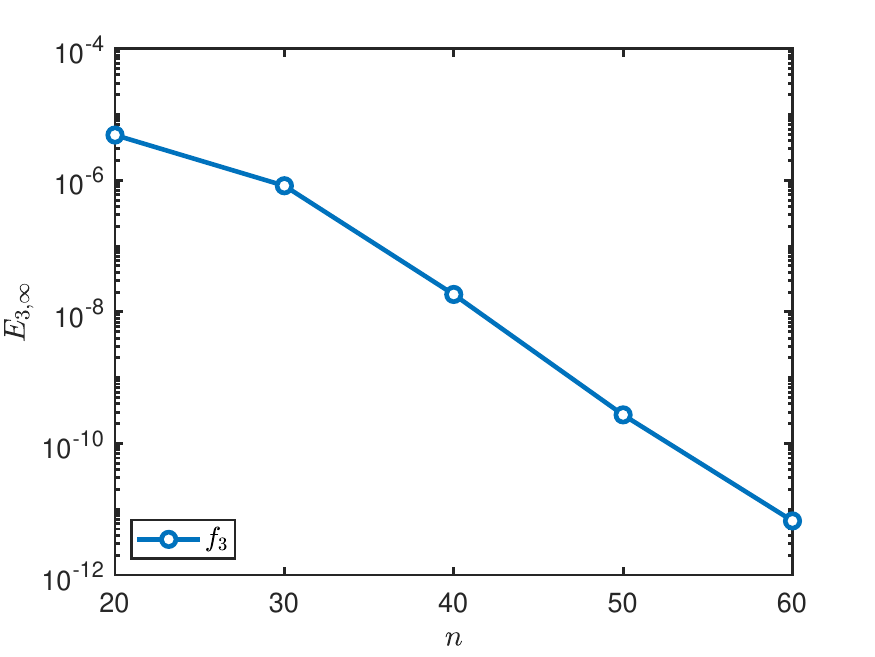}
  \includegraphics[width=0.49\textwidth]{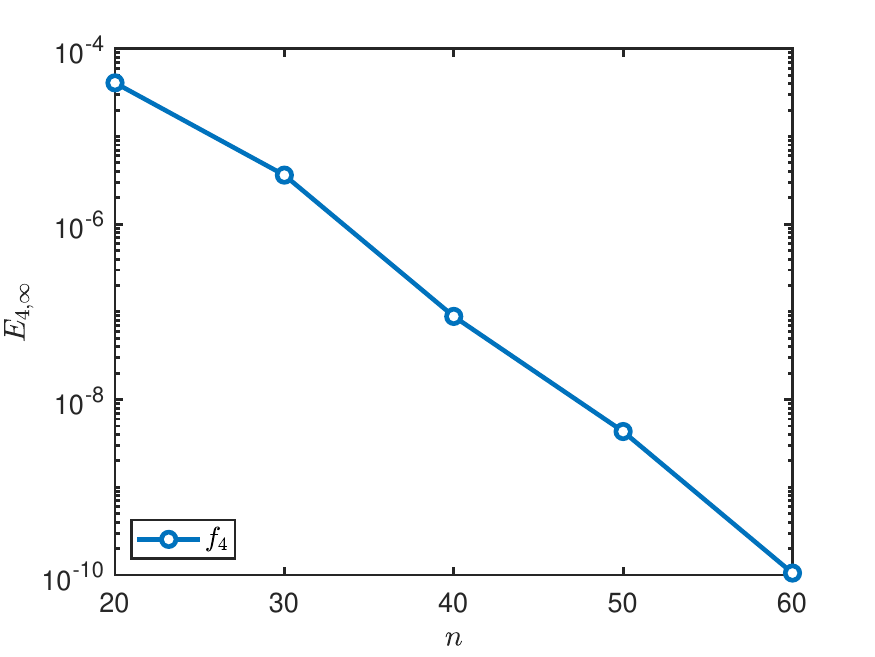}
  \caption{Trend of the maximum error $E_{i,\infty}\left(r;X_N,X_M\right)$ in the general setting for the test functions $f_1$ (top left), $f_2$ (top right), $f_3$ (bottom left), and $f_4$ (bottom right).}
  \label{fig4}
\end{figure}
From these plots, we can observe that the error decreases with $n$ for all the test functions considered, with a trend comparable to that observed in the antipodal case. Here, however, the approximation uses only $(n+1)^2$ sampling points, whereas the antipodal construction relies on $2(n+1)^2$ nodes. In these experiments the errors are smaller than in the antipodal setting. This is likely due to the quality of the approximate Fekete subset used for interpolation.

To further investigate the effect of the reconstruction degree, we consider an additional set of experiments in which 
\[
n=50, \quad m=\left\lfloor \frac{n}{4}\right\rfloor +1,
\]
while the degree $r$ is allowed to vary.  More precisely, we analyze the trend of the maximum error for the test functions $f_1,f_2,f_3,f_4$ as a function of $r$. The results are reported in Fig.~\ref{figtrend}.
\begin{figure}
  \centering
\includegraphics[width=0.49\textwidth]{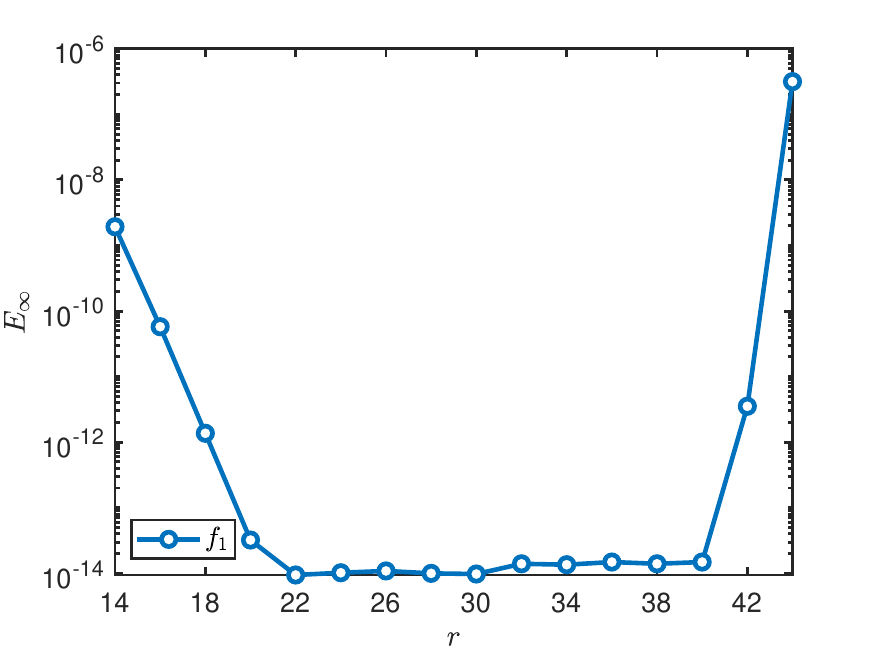} 
\includegraphics[width=0.49\textwidth]{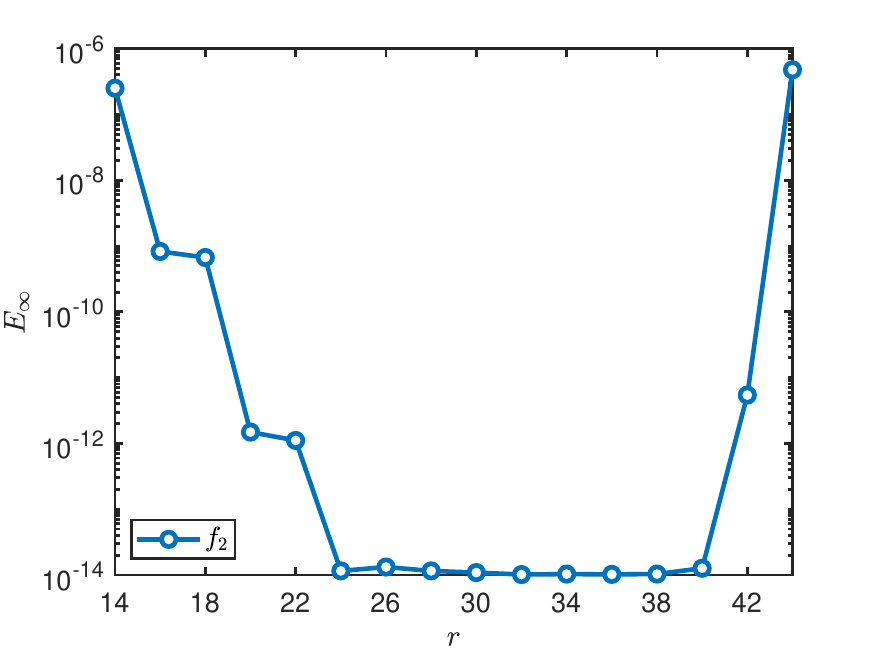} 
\includegraphics[width=0.49\textwidth]{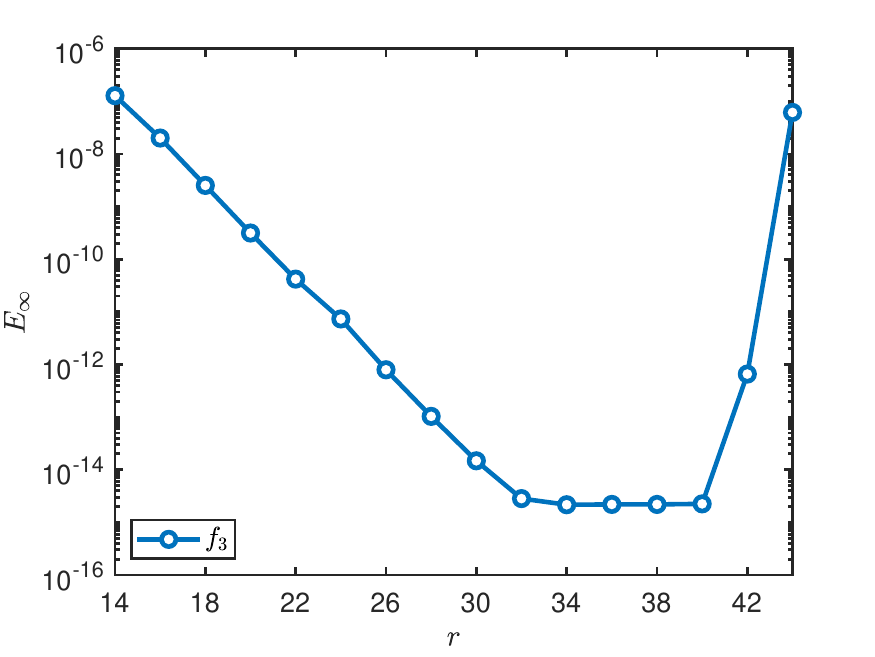} 
\includegraphics[width=0.49\textwidth]{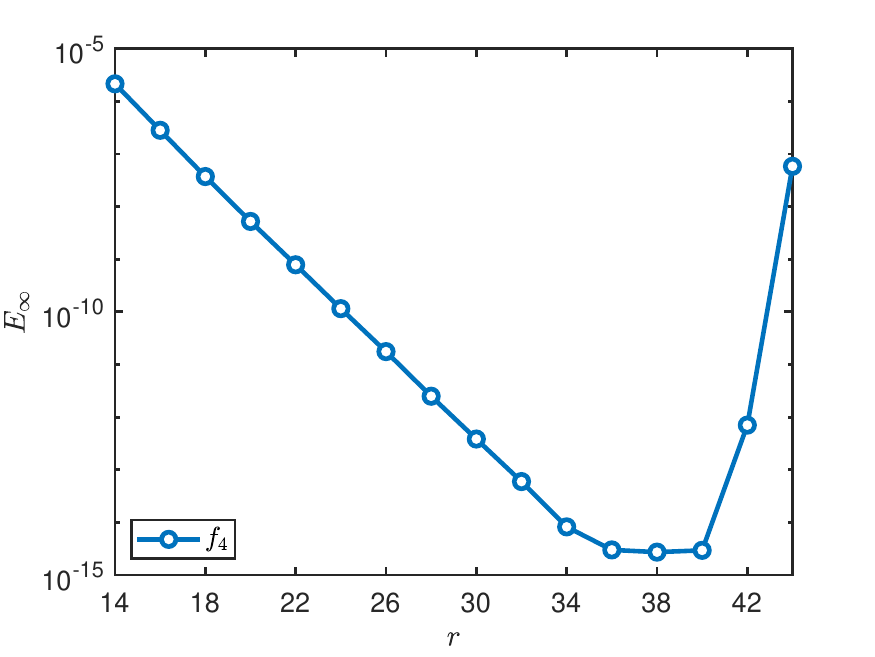} 
\caption{Trend of the maximum error $E_{i,\infty}\left(r;X_N,X_M\right)$ as $r$ varies, for the test functions $f_1$ (top left), $f_2$ (top right), $f_3$ (bottom left), and $f_4$ (bottom right) for $n=50$  and $m=\left\lfloor \frac{n}{4}\right\rfloor +1$.}
 \label{figtrend}
\end{figure}
In this case, we observe that the error decreases as $r$ increases until it reaches machine precision.  Once this value is attained, the error remains essentially constant over a range of degrees, before starting to increase for larger values of $r$. A detailed study of the optimal choice of the reconstruction degree $r$ is left for future work.

Finally, we compare the interpolation--regression approximant with two standard strategies in the same setting. The first one is polynomial interpolation of degree $m$ constructed on the approximate Fekete subset $X_M$, while the second one is the unconstrained least squares approximation of degree $r$ computed over the full sampling set $X_N$. In particular, we consider the configuration
\[
n=50, \qquad 
m=16,
\qquad
r=25.
\]
\begin{table}[h]
\centering
\caption{Comparison of the $L_\infty$ errors for interpolation on $X_M$, interpolation--regression, and global least squares for the test functions in the general setting.}
\begin{tabular}{lccc}
\hline
Function & Interpolation on $X_M$ & Interpolation--regression & Least squares on $X_N$ \\
\hline
$f_1$ & $ 8.627e-11$ & $1.044e-14$ & $5.107e-15$ \\
$f_2$ & $1.335e-09$ & $1.132e-14$ & $3.775e-15$ \\
$f_3$ & $2.849e-08$ & $6.440e-12$ & $3.667e-12$ \\
$f_4$ & $2.272e-07$ & $3.404e-11$ & $3.076e-11$ \\
\hline
\end{tabular}
\label{tab:comparison}
\end{table}
The results in Table~\ref{tab:comparison} show that the interpolation--regression approximant reaches an accuracy very close to that of the unconstrained least squares solution, while still enforcing interpolation on $X_M$. Thus, in this configuration, the interpolation constraints do not significantly affect the quality of the global fit. By contrast, interpolation on $X_M$ alone gives larger errors, mainly because it uses a lower polynomial degree and does not exploit the full sampling set.

\section{Conclusions and Future Work}
In this paper, we introduced a constrained interpolation--regression procedure for polynomial approximation on $\mathbb{S}^2$. The construction combines interpolation on a selected subset of nodes with a least squares approximation on the full sampling set. Under natural rank assumptions, the problem is well posed and the approximant admits an orthogonality characterization. In the antipodally symmetric setting, the parity of real spherical harmonics yields a decomposition into even and odd components and an associated block structure for the discrete operators. In the case of spherical designs, the normal matrix becomes a scalar multiple of the identity, so the spectral analysis simplifies considerably.  We also derived an $L^2\left(\mathbb{S}^2\right)$ quasi-optimality estimate under a discrete norming assumption.

Future work will focus on the following problems. The first is a more precise analysis of the impact of the reconstruction degree $r$ on the stability of the method. The second is the extension of the present interpolation--regression framework to more general Riemannian manifolds.

\section*{Declarations}

\textbf{Conflict of Interest}\\
The authors declare that they have no conflict of interest.\\

\noindent
\textbf{Funding statement}\\
This work was partially supported by the Institute of Mathematics of the University of Granada (IMAG), through its Support Program for Research Activities 2026. 
This research was supported by the GNCS-INdAM 2026 project 
\emph{``Metodi polinomiali e kernel per l'approssimazione da dati discreti e integrali con software OS''} 
(CUP E53C25002010001). 
The work of F. Nudo was funded by the European Union -- NextGenerationEU under the Italian National Recovery and Resilience Plan (PNRR), Mission 4, Component 2, Investment 1.2 
\lq\lq Finanziamento di progetti presentati da giovani ricercatori\rq\rq, 
pursuant to MUR Decree No.~47/2025. The work of T. E. Pérez and M. A. Piñar was funded by the grants PID2023.149117NB.I00 funded by MICIU/ AEI/
10.13039/501100011033 and ERDF A way of making
Europe.\\

\noindent
\textbf{Author Contributions}\\
All authors contributed equally to this work; therefore, the order of the authors is alphabetical. \\

\noindent
\textbf{Acknowledgement}\\
This research was carried out as part of RITA \textquotedblleft Research ITalian network on Approximation'' and as part of the UMI group \enquote{Teoria dell'Approssimazione e Applicazioni}.\\

\noindent
\textbf{Data Availability}\\
No data were used in this study.

\bibliographystyle{spmpsci}
\bibliography{bibliography}

\end{document}